\documentclass[reqno]{amsart}
\usepackage{latexsym,amsfonts,amssymb,epsfig}
\usepackage{hyperref} 
\usepackage{color}

\DeclareMathOperator{\ch}{char}

\DeclareMathOperator{\ad}{ad}
\DeclareMathOperator{\Der}{Der}

\DeclareMathOperator{\Lie}{Lie}
\DeclareMathOperator{\Alg}{Alg}

\DeclareMathOperator{\End}{End}

\DeclareMathOperator{\cwt}{\mathsf{wt}}    
\DeclareMathOperator{\wt}{wt}    
\DeclareMathOperator{\Gr}{Gr}     

\renewcommand {\limsup}{\operatorname* {\overline{lim}}}
\renewcommand {\liminf}{\operatorname* {\underline{lim}}}

\DeclareMathOperator{\GKdim}{GKdim}
\DeclareMathOperator{\LGKdim}{\underline{GKdim}}

\DeclareMathOperator{\EXP}{EXP}

\DeclareMathOperator{\Ldim}{Ldim}
\DeclareMathOperator{\LLdim}{\underline{Ldim}}

\newcommand\dd{\partial}

\renewcommand{\a}{\alpha}
\renewcommand{\b}{\beta}
\newcommand{\g}{\gamma}

\newcommand{\Z}{\mathbb Z}            
\newcommand{\R}{\mathbb R}            
\newcommand{\N}{\mathbb N}            
\newcommand{\F}{\mathbb F}            
\newcommand{\C}{\mathbb C}            
\newcommand\NO{\mathbb N_0}           

\newcommand{\LL}{\mathbf L}        
\newcommand{\QQ}{\mathbf Q}        
\newcommand{\TT}{\mathbf T}        

\renewcommand{\AA}{\mathbf A}      
\newcommand{\uu}{\mathbf u}         

\newtheorem{Theorem}{Theorem}[section]
\newtheorem{Corollary}[Theorem]{Corollary}
\newtheorem{Lemma}[Theorem]{Lemma}

\theoremstyle{remark}
\newtheorem{Remark}{Remark}
\theoremstyle{Example}
\newtheorem{Example}{Example}
\theoremstyle{Definition}
\newtheorem{Definition}{Definition}
\theoremstyle{Conjecture}

\begin{document}
\title{Clover nil restricted Lie algebras of quasi-linear growth}
\author{Victor Petrogradsky}
\address{Department of Mathematics, University of Brasilia, 70910-900 Brasilia DF, Brazil}
\email{petrogradsky@rambler.ru}
\thanks{The author was partially supported by “bolsa de produtividade em pesquisa” CNPq~309542/2016-2, Brazil, and  FAPDF 2019/01
\\ \vspace{-0.3cm}
}


\begin{abstract}
The Grigorchuk and Gupta-Sidki groups play fundamental role in modern group theory.
They are natural examples of self-similar finitely generated periodic groups.
The author constructed their analogue in case of restricted Lie algebras of characteristic 2~\cite{Pe06},
Shestakov and Zelmanov extended this construction to an arbitrary positive characteristic~\cite{ShZe08}.

Now, we construct a family of so called {\it clover } 3-generated restricted Lie algebras $\mathbf{T}(\Xi)$,
where a field of positive characteristic is arbitrary and $\Xi$ an infinite tuple of positive integers.
All these algebras have a nil $p$-mapping.
We prove that $1\le \mathrm{GKdim}\mathbf{T}(\Xi)\le 3$.
We compute Gelfand-Kirillov dimensions of clover restricted Lie algebras with periodic tuples and show
that these dimensions for constant tuples are dense on $[1,3]$.
We construct a subfamily of 
nil restricted Lie algebras $\TT(\Xi_{q,\kappa})$, with parameters $q\in \N$, $\kappa\in\R^+$,
having extremely slow {\it quasi-linear} growth of type:
$\gamma_{\mathbf{T}(\Xi_{q,\kappa})}(m)=m \big(\underbrace{\ln\cdots\ln}_q  m\big )^{\kappa+o(1)}$, as $m\to\infty$.

The present research is motivated by a construction by Kassabov and Pak of groups of oscillating growth~\cite{KasPak13}.
As an analogue, we construct nil restricted Lie algebras of intermediate oscillating growth in~\cite{Pe20flies}.
We call them {\it Phoenix algebras} because,  for infinitely many periods of time, the algebra is "almost dying" by having
a "quasi-linear" growth as above, for infinitely many $n$ it has a rather fast intermediate growth of type
$\exp( n/ (\ln n)^{\lambda})$, for such periods the algebra is "resuscitating".
The present construction of 3-generated nil restricted Lie algebras of quasi-linear growth is an important part of that result,
responsible for the lower quasi-linear bound in that construction.
\end{abstract}
\maketitle

\section{Introduction}

Different versions of Burnside Problem ask what one can say about finitely generated periodic groups under additional assumptions.
Kurosh type problems ask similar questions
about properties of finitely generated nil (more generally, algebraic) associative algebras.
In case of finitely generated Lie algebras, the periodicity is replaced by the condition that the adjoint mapping is nil.
In particular, for Lie $p$-algebras one assumes that the $p$-mapping is nil.
One of recent important directions in these areas is to study the growth of finitely generated (periodic)
groups and (nil) algebras~\cite{ErshlerZheng20,BellZel19}.
The goal of this paper is to construct finitely generated nil restricted Lie algebras
with extremely slow quasi-linear growth, these algebras are needed in further research~\cite{Pe20flies}.
Main results are formulated in Section~\ref{Smain}, see Theorem~\ref{Tparam} and Theorem~\ref{Tparam2}.

\subsection{Kurosh problem, Golod-Shafarevich algebras and groups}
The General Burnside Problem asks whether a finitely generated periodic group is finite.
The first negative answer was given by Golod and Shafarevich:
they proved that there exist finitely generated infinite $p$-groups for each prime $p$~\cite{Golod64}.
As an important instrument, they first construct finitely generated
infinite dimensional associative nil-algebras~\cite{Golod64}.
Using this construction, there are also examples of infinite dimensional 3-generated Lie algebras $L$
such that $(\ad x)^{n(x,y)}(y)=0$, for all $x,y\in L$, the field being arbitrary~\cite{Golod69}.
Similarly, one easily obtains infinite dimensional finitely generated restricted Lie algebras $L$ with  a nil $p$-mapping.
This gives a negative answer to the question of Jacobson whether
a finitely generated restricted Lie algebra $L$ is finite dimensional provided that
each element $x\in L$ is algebraic, i.e. satisfies some $p$-polynomial $f_{p,x}(x)=0$
(\cite[Ch.~5, ex.~17]{JacLie}). 

It is known that the construction of Golod yields associative nil-algebras of exponential growth.
Using specially chosen relations, Lenagan and Smoktunowicz constructed associative nil-algebras of polynomial growth~\cite{LenSmo07};
there are more constructions including associative nil-algebras of intermediate growth~\cite{BellYoung11,LenSmoYoung12,Smo14}.
On further developments concerning Golod-Shafarevich algebras and groups see~\cite{Voden09,Ershov12}.

A close by spirit but different construction was motivated by respective group-theoretic results.
A restricted Lie algebra $G$ is called {\it large} if there is a subalgebra  $H\subset G$ of finite codimension
such that $H$ admits a surjective homomorphism on a nonabelian free restricted Lie algebra.
Let $K$ be a perfect at most countable field of positive characteristic.
Then there exist infinite-dimensional finitely generated nil restricted Lie algebras over $K$ that
are residually finite dimensional and direct limits of large restricted Lie algebras~\cite{BaOl07}.

\subsection{Grigorchuk and Gupta-Sidki groups}
The construction of Golod is rather undirect, Grigorchuk gave a direct and elegant construction of
an infinite 2-group generated by three elements of order 2~\cite{Grigorchuk80}.
Originally, this group was defined as a group of transformations of the interval $[0,1]$ from which
rational points of the form $\{k/2^n\mid  0\le k\le 2^n,\ n\ge 0\}$ are removed.
For each prime $p\ge 3$, Gupta and Sidki gave a direct construction of an infinite $p$-group
on two generators, each of order $p$~\cite{GuptaSidki83}.
This group was constructed as a subgroup of an automorphism group of an infinite regular tree of degree $p$.

The Grigorchuk and Gupta-Sidki groups are counterexamples to the General Burnside Problem.
Moreover, they gave answers to important problems in group theory.
So, the Grigorchuk group and its further generalizations
are first examples of groups of intermediate growth~\cite{Grigorchuk84}, thus answering
in negative to a conjecture of Milnor that groups of intermediate growth do not exist.
The construction of Gupta-Sidki also yields groups of subexponential growth~\cite{FabGup85}.
The Grigorchuk and Gupta-Sidki groups are {\it self-similar}.
Now self-similar, and so called {\it branch groups}, form a well-established area in group theory~\cite{Grigorchuk00horizons,Nekr05}.

\subsection{Fibonacci Lie algebra, nil (restricted) Lie (super)algebras}
There are also constructions of self-similar associative  algebras~\cite{Bartholdi06,Sidki09,PeSh13ass}.
Despite some efforts~\cite{Sidki09,PeSh13ass}, in case of associative algebras, an appropriate analogue of
the Grigorchuk and Gupta-Sidki groups is not known yet.
But in case of restricted Lie algebras, we have natural analogues.
\begin{Example} ({\it Fibonacci restricted Lie algebra}~\cite{Pe06}).
Let $\ch K=p=2$ and $R=K[t_i| i\ge 0 ]/(t_i^p| i\ge 0)$, a truncated polynomial ring.
Put $\dd_i=\frac {\dd}{\partial t_i}$, $i\ge 0$.
Define two derivations of $R$:
\begin{align*}
v_1 & =\dd_1+t_0(\dd_2+t_1(\dd_3+t_2(\dd_4+t_3(\dd_5+t_4(\dd_6+\cdots )))));\\
v_2 & =\qquad\quad\;\,
\dd_2+t_1(\dd_3+t_2(\dd_4+t_3(\dd_5+t_4(\dd_6+\cdots )))).
\end{align*}
Consider the restricted Lie algebra generated by them
$\LL=\Lie_p(v_1,v_2)\subset\Der R$ and an associative algebra $\AA=\Alg(v_1,v_2)\subset \End R$.
\end{Example}
The Fibonacci restricted Lie algebra has a slow polynomial growth
with Gelfand-Kirillov dimension $\GKdim \LL=\log_{(\sqrt 5+1)/2} 2\approx 1.44$~\cite{Pe06}.
Further properties of the Fibonacci restricted Lie algebra 
are studied in~\cite{PeSh09,PeSh13fib}.
On background and some results on Lie algebras of differential
operators in infinitely many variables see~\cite{Razmyslov,Rad86,PeRaSh,FutKochSis}.

Probably, the most interesting property of $\LL$ is that it has a nil $p$-mapping~\cite{Pe06},
which is an analog of the periodicity of the Grigorchuk and Gupta-Sidki groups.
We still do not know whether the associative hull $\AA$ is a nil-algebra.
We have a weaker statement. The algebras $\LL$, $\AA$, and the augmentation ideal
of the restricted enveloping algebra $\uu=\omega u(\LL)$ are direct sums of two locally nilpotent subalgebras~\cite{PeSh09}.

The next step was made by Shestakov and Zelmanov,
in case of an arbitrary prime characteristic,
they constructed an example of a 2-generated restricted Lie algebra with a nil $p$-mapping~\cite{ShZe08}.
An example of a $p$-generated  nil restricted Lie algebra $L$, characteristic $p$ being arbitrary, was studied in~\cite{PeShZe10}.
These infinite dimensional restricted Lie algebras and their restricted enveloping algebras as well, have
different decompositions into a direct sum of two locally nilpotent subalgebras~\cite{PeShZe10}.

Observe that only the original example has a clear monomial basis~\cite{Pe06,PeSh09}.
In other examples, elements of a Lie algebra are  linear combinations of monomials;
to work with such linear combinations is sometimes an essential technical difficulty, see e.g.~\cite{ShZe08,PeShZe10,Pe20flies}.

A systematic approach to construct (restricted) Lie (super)algebras
having {\it good monomial bases} was developed due to the second Lie superalgebra introduced in~\cite{Pe16}.
\begin{Example}\label{Example_Q}(second example $\QQ$ in \cite{Pe16})
Consider the Grassmann superalgebra $\Lambda=\Lambda[x_i,y_i,z_i| i\ge 0]$, field being arbitrary.
Using its partial superderivations,
define recursively odd elements in the associative superalgebra $\End(\Lambda)$:
\begin{equation*}
\begin{split}
a_i &= \partial_{x_i} + y_ix_i a_{i+1},\\
b_i &= \partial_{y_i} + z_iy_i b_{i+1},\\
c_i &= \partial_{z_i} + x_iz_i c_{i+1},
\end{split}
\qquad i\ge 0.
\end{equation*}
Define the Lie superalgebra $\QQ:=\Lie(a_0,b_0,c_0)\subset \Der\Lambda$.
\end{Example}
Using a similar approach, a family of nil restricted Lie algebras of slow polynomial growth having  good monomial bases
was constructed in~\cite{Pe17} (Example~\ref{E2}, actually, these algebras are more close to the first example in~\cite{Pe16}).

Informally speaking, there are no "natural analogues" of the Grigorchuk group
in the world of Lie algebras of characteristic zero~\cite{MaZe99}.
On the other hand, we show that Example~\ref{Example_Q} serves as an appropriate analogue of the Grigorchuk group in the class of Lie {\it super}algebras
in case of an arbitrary field, because it is nil finely $\Z^3$-graded, see details in~\cite{Pe16}.
Next, we construct a more "handy" 2-generated fractal Lie superalgebra $\mathbf{R}$
over an arbitrary field~\cite{PeOtto}.
This example is close to the smallest possible one, because $\mathbf{R}$ has a linear growth
with a growth function $\gamma_\mathbf{R}(m)\approx 3m$, as $m\to\infty$.
Moreover, its degree $\mathbb{N}$-grading is of finite width 4 ($\ch K\ne 2$).
We also construct a  just infinite fractal 3-generated Lie superalgebra ${\mathbf Q}$ over an arbitrary field,
which gives rise to an associative hull, a Poisson superalgebra, and two Jordan superalgebras supplying
analogues of the Grigorchuk and Gupta-Sidki groups in respective classes of algebras~\cite{PeSh18FracPJ}.

\subsection{Narrow groups and Lie algebras}
The Grigorchuk group $G$ is of finite width, namely,
the lower central series factors are bounded~\cite{Rozh96,BaGr00,Grigorchuk00horizons}.
In particular, the respective Lie algebra $L=L_K(G)=\oplus_{i\ge 1} L_i$ has a linear growth.
Bartholdi presented $L_{K}(G)$ as a self-similar restricted Lie algebra
and proved that the restricted Lie algebra $L_{\F_2}(G)$ is nil while $L_{\F_4}(G)$ is not nil~\cite{Bartholdi15}.
Also, $L_K(G)$ is {\it nil graded}, namely,
for any homogeneous element $x\in L_i$, $i\ge 1$, the mapping $\ad x$ is nilpotent, because the group $G$ is periodic.
Naturally $\N$-graded Lie algebras over $\R$ and $\C$
satisfying the condition $\dim L_n+\dim L_{n+1}\le 3$, $n\ge 1$, are classified recently by Millionschikov~\cite{Mil20}.
Slowly growing so called filiform Lie algebras in characteristic zero are studied in~\cite{CarMatNew97,CarNew00}.
Concerning narrow Lie algebras and groups see survey~\cite{ShaZel99}.

\section{Basic notions: restricted Lie algebras, Growth}\label{Sdef}

As a rule, $K$ is an arbitrary field of positive characteristic $p$,
$\langle S\rangle_K$ denotes a linear span of a subset $S$ in a $K$-vector space.
Let $L$ be a Lie algebra, then $U(L)$ denotes the universal enveloping algebra.
Long commutators are {\it right-normed}: $[x,y,z]:=[x,[y,z]]$.
We use a standard notation $\ad x(y)=[x,y]$, where $x,y\in L$.
Also, we use the notation $[x^k,y]:=(\ad x)^k (y)$, where $k\ge 1$, $x,y\in L$; in case $k=p^l$, we have also $[x^{p^l},y]=[x^{[p^l]},y]$,
in terms of the $p$-mapping (see below).

\subsection{Restricted Lie algebras}
Let $L$ be a Lie algebra over a field $K$ of characteristic $p>0$.
Then $L$ is called a
\textit{restricted Lie algebra} (or \textit{Lie $p$-algebra}),
if it is additionally supplied with a unary operation
 $x\mapsto x^{[p]}$, $x\in L$, that satisfies the following
 axioms~\cite{JacLie,Ba,Strade1,StrFar,BMPZ}:
\begin{itemize}
\item $(\lambda x)^{[p]}=\lambda^px^{[p]}$, for $\lambda\in K$, $x\in L$;
\item $\ad(x^{[p]})=(\ad x)^p$, $x\in L$;
\item $(x+y)^{[p]}=x^{[p]}+y^{[p]}+\sum_{i=1}^{p-1}s_i(x,y)$, for all $x,y\in L$,
where $i s_i(x,y)$~is the coefficient of $t^{i-1}$ in the polynomial
$\operatorname{ad}(tx+y)^{p-1}(x)\in L[t]$.
\end{itemize}
This notion is motivated by the following construction.
Let $A$ be an associative algebra over a field ~$K$.
If the vector space $A$ is supplied with a new product $[x,y]=xy-yx$, $x,y\in A$,
one obtains a Lie algebra denoted by $A^{(-)}$.
In case $\operatorname{char}K=p>0$,
the mapping  $x\mapsto x^p$, $x\in A^{(-)}$, satisfies the three axioms above.

Suppose that $L$~ is a restricted Lie algebra.
Let $J$~be the ideal of the universal enveloping algebra~$U(L)$ generated by $\{x^{[p]}-x^p\mid x\in L\}$.
Then $u(L)=U(L)/J$ is called a \textit{restricted enveloping algebra}.
In this algebra, the formal operation $x^{[p]}$ coincides with the $p$th power~$x^p$ for any $x\in L$.
One has an analogue of Poincare-Birkhoff-Witt's theorem
yielding a basis of the restricted enveloping algebra~\cite[p.~213]{JacLie}.
We shall use the following version of the formula above:
\begin{equation}\label{power_P}
(x+y)^{[p]}=x^{[p]}+y^{[p]}+(\ad x)^{p-1}(y)+
\sum_{i=1}^{p-2}s_i(x,y),\qquad x,y\in L,
\end{equation}
where $s_i(x,y)$ consists of commutators containing $i$ letters $x$ and $p-i$ letters $y$.

\subsection{Growth}
Let $A$  be an associative (or Lie) algebra  generated by a finite set $X$.
Denote  by $A^{(X,n)}$ the subspace of $A$ spanned by all  monomials  in $X$ of length not  exceeding  $n$, $n\ge 0$.
If $A$ is a restricted Lie algebra, we define
$A^{(X,n)}=\langle\, [x_{i_1},\dots,x_{i_s}]^{p^k}\mid x_{i_j}\in X,\, sp^k\le n\rangle_K$~\cite{Pape01}.
One obtains the {\em growth function}:
$$
\gamma_A(n)=\gamma_A(X,n):=\dim_KA^{(X,n)},\quad n\ge 0.
$$
Clearly, the growth function depends on the choice of the generating set $X$.
Let $f,g:\N\to\R^+$ be increasing functions.
Write $f(n)\preccurlyeq g(n)$ if and only if there exist positive integers $N,C$ such that $f(n)\le g(Cn)$ for all $n\ge N$.
Introduce equivalence $f(n)\sim g(n)$ if and only if  $f(n)\preccurlyeq g(n)$ and $g(n)\preccurlyeq f(n)$.
Different generating sets of an algebra yield equivalent growth functions~\cite{KraLen}.

The growth of the free associative, (restricted) Lie (super)algebras is exponential~\cite{Ba,BMPZ,KraLen,Pe03}.
Moreover, any finitely generated linear algebra (i.e. existence of a bilinear product is assumed only)
has at most exponential growth, because it is a
homomorphic image of the {\it absolutely free algebra} with finite number of generators,
which exponential growth is well-known, see e.g.~\cite{Pe05}.
To describe the growth of a finitely generated algebra $A$ one defines its {\it exponent}
(which depends on the generating set!):
$$\EXP(A,X):=\limsup_{n\to \infty} \sqrt[n]{\gamma_A(X,n)},$$
where in case of an associative algebra $A$ there exists the two-sided limit~\cite{KraLen}.
If $\EXP(A)>1$ then $A$ is said of {\it exponential growth}.
Otherwise, $\EXP A=1$, and $A$ is said of {\it subexponential growth}.
If there exists a constant $\a>0$ such that $\gamma_A(n)\preccurlyeq n^\a$, then $A$ has a {\it polynomial growth}.
A subexponential growth that is not polynomial is called {\it intermediate}.
These types of growth do not depend on the generating set.

A growth function $\gamma_A(n)$ is compared with polynomial functions $n^\a$, $\a\in\R^+$,
by computing the {\em upper and lower Gelfand-Kirillov dimensions}~\cite{KraLen}:
\begin{align*}
\GKdim A&:=\limsup_{n\to\infty} \frac{\ln\gamma_A(n)}{\ln n}=\inf\{\a>0\mid \gamma_A(n)\preccurlyeq n^\a\} ;\\
\LGKdim A&:=\liminf_{n\to\infty}\,  \frac{\ln\gamma_A(n)}{\ln n}=\sup\{\a>0\mid \gamma_A(n)\succcurlyeq n^\a\}.
\end{align*}
Solvable finitely generated Lie algebras have subexponential growth, see~\cite{Licht84}.
The author constructed a scale to measure the growth of such algebras~\cite{Pe96,Pe99int}.

\subsection{Quasi-linear growth, its stratification}
Now, assume that $A$ is an associative or Lie algebra generated by a finite set $X$.
Consider the sequence of non strictly increasing subspaces $\{A^{(X,n)}| n\ge 0\}$.
There are two cases.
1) There exists $n_0\in \N$ such that $A^{(X,n_0)}=A^{(X,n_0+1)}$.
Using that all Lie monomials are expressed  via the right-normed ones, we get $A^{(X,m)}=A^{(X,n_0)}$ for all $m\ge n_0$.
Thus, $A$ is finite dimensional and $\GKdim A=0$.
2) All subsequent terms are different, hence their dimensions are strictly increasing.
By induction, we get the lower bound
$\gamma_A(n,X)=\dim A^{(X,n)}\ge n+1$ and $\GKdim A\ge 1$.
Thus, one has a trivial gap $\GKdim A\notin (0,1)$.
Also, if $A$ is infinite dimensional, then
the growth function is bounded from below by the linear function $n+1$.
The analogue of this gap for restricted Lie algebras is studied in~\cite{Pape01}.

In this paper, we construct algebras whose growth is somewhat close to that lowest possible linear growth function.
Let an algebra $A$  satisfies $\GKdim A=\LGKdim A=1$, we say that $A$  has a {\it quasi-linear growth}.
Quasi-linear growths are
not distinguishable from viewpoint of the Gelfand-Kirillov dimension because they merge as "one point".
In order to blow up this point, we compare a quasi-linear growth function with two families of etalon functions.
Denote $\ln^{(q)}(x):=\underbrace{\ln(\cdots\ln}_{q\text{ times}}(x)\cdots)$ and
$\exp^{(q)}(x):=\underbrace{\exp(\cdots\exp}_{q\text{ times}}(x)\cdots)$ for all $q\in\N$.
Consider the first family of quasi-linear functions:
$m\exp \big((\ln m)^{\beta}\big)$, $\beta\in(0,1)$ being a constant.
The second family of quasi-linear functions is
$m (\ln^{(q)} m)^{\beta}$, where $q\in \N$, $\beta\in \R^+$ are constants.
Now, we compare a growth function with these etalon functions, determining their parameters $q,\beta$. Formally, we set:
\begin{align*}
\Ldim^0 A=& \inf\{\beta\in(0,1) \mid \gamma_A(n) \preccurlyeq m \exp \big((\ln m)^{\beta}\big)\}, \quad (q=0);\\
\LLdim^0 A=& \sup\{\beta\in(0,1) \mid \gamma_A(n) \succcurlyeq m \exp \big((\ln m)^{\beta}\big)\}, \quad (q=0);\\
\Ldim^q A=& \inf\{\beta\in\R^+ \mid \gamma_A(n) \preccurlyeq m (\ln^{(q)} m)^{\beta}\},\qquad\qquad\ q\in\N;\\
\LLdim^q A=& \sup\{\beta\in\R^+ \mid \gamma_A(n) \succcurlyeq m (\ln^{(q)} m)^{\beta}\},\qquad\qquad\ q\in\N;
\end{align*}
where the last two numbers are defined for any fixed $q\in \N$,
the latter specifying the number of iterations of the logarithm in the right hand side functions.
One checks that these numbers are invariants not depending on a generating set.
We refer to $q\ge 0$ as the {\it level} of the functions above.
Remark that these notations are different from~\cite{Pe17}.
Define {\it extreme values}, in case $q=0$: $\Ldim^0 A=1$, or $\LLdim^0 A=0$;
in case $q\ge 1$: $\Ldim^q A=+\infty$, or $\LLdim^q A=0$.
In these cases, the etalon functions of level $q$ are not suited to specify the quasi-linear growth of an algebra.
Observe that the functions with bigger $q$ are slower.
One checks that the functions of different levels stratify the merged point of all quasi-linear growths as follows.
\begin{Lemma}
Assume that for some $q\ge 1$ one has $\Ldim^q A=\LLdim^q A=\beta$ where $\beta\ne 0$ and $\beta\ne +\infty$.
Then $\Ldim^{q+1} A=+\infty$ and $\LLdim^{q-1} A=0$.
\end{Lemma}

Assume that generators $X=\{x_1,\dots,x_k\}$ are assigned positive weights $\wt(x_i)=\lambda_i$, $i=1,\dots,k$.
Define a {\it weight growth function}:
$$
\tilde \gamma_A(n)=\dim_K\langle x_{i_1}\cdots x_{i_m}\mid \wt(x_{i_1})+\cdots+\wt(x_{i_m})\le n,\
          x_{i_j}\in X\rangle_K,\quad n\ge 0.
$$
(Where in case of (restricted) Lie algebras one considers (restricted) Lie monomials as above).
Set $C_1=\min\{\lambda_i\mid i=1,\dots,k \}$, $C_2=\max\{\lambda_i\mid i=1,\dots,k \}$,
then $\tilde\gamma_A(C_1 n) \le \gamma_A(n)\le \tilde\gamma_A(C_2 n)$ for $n\ge 1$.
Thus, we obtain an equivalent growth function $\tilde \gamma_A(n)\sim\gamma_A(n)$.
Therefore, we can use the weight growth function $\tilde\gamma_A(n)$ in order to
compute the Gelfand-Kirillov dimensions and $\Ldim^q A$, $\LLdim^q A$ as well.

Suppose that $L$ is a Lie algebra and $X\subset L$.
By $\Lie(X)$ denote the subalgebra of $L$ generated by $X$.
In case $L$ is a restricted Lie algebra $\Lie_p(X)$ denotes the restricted subalgebra of $L$ generated by $X$.
Similarly, assume that $X$ is a subset in  an associative algebra $A$.
Write $\Alg(X)\subset A$ to denote the associative subalgebra (without unit) generated by~$X$.

\subsection{Divided power algebra and its derivations}
Fix $\ch K=p>0$.
Let $\Theta$  be an arbitrary non-empty set.
Fix a tuple of integers $\bar S=\{S_a\in \N | a\in \Theta \}$.
We consider a {\it divided power algebra} $R=R(\Theta,\bar S)$ which $K$-basis consists of formal symbols
$$\bigg\{ \prod_{a\in \Theta}t_a^{(i_a)} \ \bigg|\ 0\le i_a< p^{S_a}, \ a\in \Theta\bigg\}, $$
where only finitely many formal powers $i_a$ are non-zero.
Define a product of these elements as
$$\bigg(\prod_{a\in \Theta}t_a^{(i_a)}\bigg )\cdot \bigg (\prod_{a\in \Theta}t_a^{(j_a)}\bigg)=
 \prod_{a\in \Theta}\binom{i_a+j_a}{i_a} t_a^{(i_a+j_a)}. $$
The product is well defined and
$R=R(\Theta,\bar S)$ is an associative commutative ring with unit, which
is isomorphic to a ring of truncated polynomials~\cite{Strade1}.

Fix $a\in \Theta$. Define an action $\partial_{a}$ on the whole of $R$ acting on
the respective divided variables only:
$\partial_{a}(t_a^{(i_a)}):=t_a^{(i_a-1)}$, $i_a\in\{0,\dots,p^{S_a}-1\}$,
where $t_a^{(0)}=1$ and $t_a^{(l)}=0$ for $l<0$.
We obtain derivations $\partial_{a}\in \Der R$, $a\in \Theta$.
Their $p^m$-powers are also derivations and
$\partial_{a}^{p^m}(t_a^{(i_a)})=t_a^{(i_a-p^m)}$, 
$m\ge 0.$
Clearly, $\partial_{a}^{p^{m}}=0$ for $m\ge S_a$.
For more properties of divided power algebras and their derivations, see~\cite{Strade1}.

\section{Main results: Clover restricted Lie algebras of Quasi-linear Growth}
\label{Smain}
\subsection{Clover restricted Lie algebras}
Recently, the author introduced a large class of {\it drosophila Lie algebras}~\cite{Pe20flies},
that yields a uniform generalized construction including
some examples of (restricted) Lie (super)algebras considered before~\cite{Pe17,Pe16}.
In particular, it includes a family of 2-generated restricted Lie algebras studied in~\cite{Pe17};
now we call such algebras as {\it duplex Lie algebras}.

\begin{Example}[family of restricted Lie algebras $\LL(\Xi)$  in~\cite{Pe17}]\label{E2}
Let $\ch K=p>0$.
Let $\Theta=\{x_n,y_n|n\ge 0\}$ and consider a tuple of integers $\Xi=(S_n,R_n| n\ge 0)$.
As described above, these parameters yield the divided power algebra with a basis
$\Omega(\Xi):=\langle x_0^{(\xi_0)}y_0^{(\eta_0)}\!\!\!\cdots x_i^{(\xi_i)}y_i^{(\eta_i)}|
0\le\xi_i<p^{S_i}, 0\le \eta_i<p^{R_i}, i\ge 0\rangle_K$. Define {\sc pivot elements} belonging to $\Der \Omega(\Xi)$ recursively:
\begin{equation}\label{aibip}
\begin{split}
a_i &= \partial_{x_i} + x_i^{(p^{S_i}-1)} y_i^{(p^{R_i}-1)} a_{i+1};\\
b_i &= \partial_{y_i} + x_i^{(p^{S_i}-1)} y_i^{(p^{R_i}-1)} b_{i+1};
\end{split}
\qquad\quad i\ge 0.
\end{equation}
Define the 2-generated restricted Lie algebra $\LL(\Xi):=\Lie_p(a_0,b_0)\subset \Der \Omega(\Xi)$.
\end{Example}

Now we define the main object to study in the present work.
\begin{Definition}\label{Def1}
Now, let $\Theta=\{x_n,y_n,z_n|n\ge 0\}$.
Fix the same tuple of integers $\Xi=(S_n,R_n| n\ge 0)$ as above and consider another
divided power algebra with a basis
$$R=R(\Xi):=\Big\langle x_0^{(\a_0)}y_0^{(\b_0)}z_0^{(\gamma_0)}\!\!\cdots x_i^{(\a_i)}y_i^{(\b_i)}z_i^{(\gamma_i)}\ \Big |\
0\le\a_i<p^{S_i},\ 0\le \b_i,\gamma_i<p^{R_i},\ i\ge 0\Big\rangle_K. $$
We draw attention that pairs of divided variables $y_i$, $z_i$ have the same top indices determined by $R_i$, for all $i\ge 0$.
This trick is important for further computations to be feasible at all.
We define recursively the  {\sc pivot elements} belonging to $\Der R(\Xi)$:
\begin{equation}\label{pivot-3}
\begin{split}
v_i &=\dd_{x_i}+x_{i}^{(p^{S_{i}}-1)} y_{i}^{(p^{R_{i}}-1)} v_{i+1} ;\\
w_i &=\dd_{y_i}+y_{i}^{(p^{R_{i}}-1)} x_{i}^{(p^{S_{i}}-1)} w_{i+1};\\
u_i &=\dd_{z_i}+z_{i}^{(p^{R_{i}}-1)} x_{i}^{(p^{S_{i}}-1)} u_{i+1};
\end{split}\qquad\qquad i\ge 0.
\end{equation}
For any $i\ge 0$, we get explicit formulas:
\begin{equation}\label{aibi0}
\begin{split}
v_i &= \partial_{x_i} {+} x_i^{(p^{S_i}-1)}y_i^{(p^{R_i}-1)}\Big(\partial_{x_{i+1}}{+} x_{i+1}^{(p^{S_{i+1}}-1)} y_{i+1}^{(p^{R_{i+1}}-1)}
       \Big(\partial_{x_{i+2}} {+}x_{i+2}^{(p^{S_{i+2}}-1)} y_{i+2}^{(p^{R_{i+2}}-1)}\Big(\cdots  \Big)\Big)\Big),\\
w_i &= \partial_{y_i} {+} y_i^{(p^{R_i}-1)}x_i^{(p^{S_i}-1)}\Big(\partial_{y_{i+1}}{+} y_{i+1}^{(p^{R_{i+1}}-1)}x_{i+1}^{(p^{S_{i+1}}-1)}
       \Big(\partial_{y_{i+2}} {+}y_{i+2}^{(p^{R_{i+2}}-1)}x_{i+2}^{(p^{S_{i+2}}-1)} \Big(\cdots  \Big)\Big)\Big),\\
u_i &= \partial_{z_i} {+} z_i^{(p^{R_i}-1)}x_i^{(p^{S_i}-1)}\Big(\partial_{z_{i+1}}{+} z_{i+1}^{(p^{R_{i+1}}-1)}x_{i+1}^{(p^{S_{i+1}}-1)}
       \Big(\partial_{z_{i+2}} {+}z_{i+2}^{(p^{R_{i+2}}-1)}x_{i+2}^{(p^{S_{i+2}}-1)} \Big(\cdots  \Big)\Big)\Big).
\end{split}
\end{equation}
We call $i$ the {\sc length} (also {\sc generation}, following terminology~\cite{Pe20flies}) of the pivot elements above.
Now,  we define the 3-generated {\sc clover restricted Lie algebra} $\TT(\Xi):=\Lie_p(v_0,w_0,u_0)\subset \Der R(\Xi)$.
\end{Definition}

\begin{Remark}
Let us draw attention that there is some symmetry between $v_i$ and $w_i$,
while the remaining $u_i$ stays separate because there is {\bf no $\Z_3$-cyclic symmetry}
unlike the second example of a 3-generated Lie superalgebra in~\cite{Pe16}.
Another observation is that $v_i$, $w_i$ are just renaming of $a_i$, $b_i$ in~\eqref{aibip}, for all $i\ge 0$.
\end{Remark}

\begin{Remark}
Example~\eqref{aibip} cannot supply the lower estimate for Lie algebras of oscillating growth constructed in~\cite{Pe20flies}
and it was necessary to modify that example.
This modification is very specific in order to make the computation feasible at all.
In terminology of~\cite{Pe20flies}, species of flies having two flies in some generation either have two flies
in all subsequent generations or go extinct.
The goal in introducing the clover species is to have three flies in each generation
(yielding respective three pivot elements~\eqref{pivot-3}),
so that at some moment three flies can produce a wild specie and the constructed Lie algebra
can return to a fast intermediate growth.
To this end we extend the duplex specie in a specific "skew" way and obtain the clover species.
This idea enables us to construct restricted Lie algebras with an oscillating growth in~\cite{Pe20flies}
using the two theorems below.
\end{Remark}

\subsection{Main results}
As a specific case, we construct restricted Lie algebras of quasi-linear growth.
The main goal of the paper is to prove the following two theorems, which
are an important part of the construction of nil restricted Lie algebras of oscillating intermediate growth in~\cite{Pe20flies},
namely, the algebras constructed below are responsible for periods of quasi-linear growth of that algebras.
We stress that it was necessary to change the approach of~\cite{Pe17}, because in a further construction
of nil restricted Lie algebras of oscillating growth~\cite{Pe20flies} we need three so called "flies" in each generation.
The asymptotic in~\cite{Pe17} has the upper and lower bounds with different constants $C_1$, $C_2$.
Now we are proving a stronger asymptotic with bounds $C+o(1)$, constant being the same for both sides.
Moreover, the second theorem yields even slower quasi-linear growths.

\begin{Theorem}\label{Tparam}
Let $K$ be a field, $\ch K=p> 0$, fix $\kappa\in(0,1)$.
There exists a tuple of integers $\Xi_\kappa$ such that
the 3-generated clover restricted Lie algebra $\TT=\TT(\Xi_\kappa)=\Lie_p(v_0,w_0,u_0)$ has the following properties.
\begin{enumerate}
\item
$\gamma_{\TT}(m)=m\exp \big((C+o(1))(\ln m)^\kappa\big)$ as $m\to\infty$, where $C:=2(\ln p)^{1-\kappa}/\kappa^\kappa$;
\item $\GKdim \TT=\LGKdim \TT= 1$;
\item $\Ldim^0 \TT=\LLdim^0 \TT=\kappa$;
\item the growth function $\gamma_\TT(m)$ is not linear;
\item algebras $\TT(\Xi_\kappa)$ for different parameters $\kappa\in(0,1)$ are not isomorphic.
\end{enumerate}
\end{Theorem}

In comparison with~\cite{Pe17},  algebras with even slower quasi-linear growth are constructed in the next theorem.

\begin{Theorem}\label{Tparam2}
Let $\ch K=p> 0$, fix parameters $q\in\N$, $\kappa\in\R^+$.
There exists a tuple of integers $\Xi_{q,\kappa}$ such that
the 3-generated clover restricted Lie algebra $\TT=\TT(\Xi_{q,\kappa})=\Lie_p(v_0,w_0,u_0)$ has the following properties.
\begin{enumerate}
\item
$\gamma_{\TT}(m)=
m \big(\ln^{(q)} \!m\big )^{\kappa+o(1)}$ as $m\to\infty$;
\item
$\GKdim \TT=\LGKdim \TT= 1$;
\item $\Ldim^q \TT=\LLdim^q \TT=\kappa$;
\item
the growth function $\gamma_\TT(m)$ is not linear;
\item algebras $\TT(\Xi_{q,\kappa})$ for different pairs $(q,\kappa)$ are not isomorphic.
\end{enumerate}
\end{Theorem}
\begin{Remark}
Similar to~\cite{Pe17}, we can also consider the associative algebra $\AA=\Alg(v_0,w_0,u_0)\subset\End R(\Xi)$ and describe its growth as
$\gamma_{\AA}(m)= m^2 \big(\ln^{(q)} \!m\big )^{\kappa+o(1)}$, as $m\to\infty$.
In particular, $\GKdim \AA=\LGKdim \AA=2$, let us  call such a growth {\it quasi-quadratic}.
\end{Remark}

\begin{Theorem}[\cite{Pe20flies}, Theorem~7.7]\label{Tnillity}
Fix $\ch K=p> 0$ and a tuple of integers $\Xi=(S_n,R_n| n\ge 0)$.
Consider the respective clover restricted Lie algebra  $\TT=\TT(\Xi)$.
Then $\TT$ has a nil $p$-mapping.
\end{Theorem}
\begin{proof}
The nillity is established in a more general setting of so called drosophila Lie algebras with uniform parameters
in~\cite{Pe20flies}, Theorem~7.7.
That proof is an essential modification of the approach in case of 2-generated duplex Lie algebras
$\LL(\Xi)$~\cite{Pe17}, Theorem~8.6.
\end{proof}
Let us describe some more results and ideas of the paper.

\begin{itemize}
\item
We describe the structure and construct a clear monomial basis for all clover restricted Lie algebras (Theorem~\ref{Tsemidirect}).
\item
An important instrument is a notion of a weight function,
using which we prove that $\TT(\Xi)=\TT(v_0,w_0,u_0)$ is $\NO^3$-graded by multidegree in the generators (Theorem~\ref{Tgraded}).
\item
We prove that $1\le \GKdim\TT(\Xi)\le 3$ for any tuple $\Xi$ (Theorem~\ref{Tgrowth3}.)
\item
Let the sequence $\Xi=(S_i,R_i|i\ge 0)$ be periodic,
then $\TT(\Xi)$ is a self-similar restricted Lie algebra (Lemma~\ref{Lself})
and we compute its Gelfand-Kirillov dimension explicitly (Theorem~\ref{Tperiod}).
This result may be viewed as an analogue of the result on the intermediate growth
for the Grigorchuk periodic groups $G_\omega$ having a periodic tuple $\omega$~\cite[Theorem~B]{ErshlerZheng20}. 
\item
Let $\Xi$ be {constant}: $S_i=S$, $R_i=R$ for $i\ge 0$, where $S,R\in\N$ are fixed.
Then denote $\TT(S,R):=\TT(\Xi)$. We prove that
$\{\GKdim\TT(S,R)\mid S,R\in\N\}$ is dense on $[1,3]$ (Corollary~\ref{Cinterval}).
\end{itemize}
\begin{Remark}
We suggest that the clover restricted Lie algebras $\TT(\Xi)$ are Lie algebra
analogues of the family of the  Grigorchuk groups $G_\omega$ constructed and studied in~\cite{Grigorchuk84}.
\end{Remark}
\subsection{Nil Lie algebras of slow polynomial growth}
The Gelfand-Kirillov dimension of an associative algebra cannot belong to the interval $(1,2)$~\cite[Bergman]{KraLen}.
One has the same gap for finitely generated Jordan algebras~\cite[Martinez and Zelmanov]{MaZe96}.
The author showed that a similar gap does not exist for Lie algebras, the Gelfand-Kirillov
dimension of a finitely generated Lie algebra can be an arbitrary number $\{0\}\cup [1,+\infty)$~\cite{Pe97}.
The same fact is also established for Jordan superalgebras~\cite{PeSh18Jslow}.
Also, an interesting direction of research is constructing associative nil algebras
of different kinds of growth, in particular, of slow polynomial growth, see~\cite{LenSmo07,BellYoung11,LenSmoYoung12,Smo14}.

Now we get a stronger version of~\cite{Pe97}: the gap $(1,2)$ can be filled with {\it nil} Lie $p$-algebras.
Namely, using constant tuples, we get self-similar clover nil restricted Lie algebras which
Gelfand-Kirillov dimensions are dense on $[1,3]$ (Corollary~\ref{Cinterval}).

\section{Structure of Clover Lie algebras $\TT(\Xi)$}
\subsection{Basic relations} We start with establishing basic relations in clover restricted Lie algebras.
In what follows, we assume that  a field $K$ of characteristic $\ch K=p>0$ and a tuple of integers
$\Xi=(S_n,R_n| n\ge 0)$ are fixed, and we consider 3-generated clover restricted Lie algebra $\TT=\Lie_p(v_0,w_0,u_0)$.

\begin{Lemma}\label{Lrelations}
Let $i\ge 0$. Then
\begin{align*}\label{vap}
v_i^{p^m}&=\dd_{x_i}^{p^m}+ x_i^{(p^{S_i}-p^m)} y_i^{(p^{R_i} -1)}v_{i+1}, \qquad\  0\le m\le S_i;\\
w_i^{p^m}&=\dd_{y_i}^{p^m}+ y_i^{(p^{R_i}-p^m)} x_i^{(p^{S_i} -1)}w_{i+1}, \qquad\  0\le m\le R_i;\\
u_i^{p^m}&=\dd_{z_i}^{p^m}+ z_i^{(p^{R_i}-p^m)} x_i^{(p^{S_i} -1)}u_{i+1},\ \qquad\  0\le m\le R_i,
\end{align*}
where $\dd_{x_i}^{p^{S_i}}=\dd_{y_i}^{p^{R_i}}=\dd_{z_i}^{p^{R_i}}=0$ above.
\end{Lemma}
\begin{proof}
Let us prove the first equality by induction on $m$. The base of induction $m=0$ is trivial by~\eqref{pivot-3}.
Assume that the claim is valid for $0\le m< S_i$.
The summation in~\eqref{power_P} is trivial because the second term cannot be used more than once:
\begin{align*}
v_i^{p^{m+1}} &={(v_i^{p^{m}})}^p=\Big(\dd_{x_i}^{p^m}+ x_i^{(p^{S_i}-p^m)} y_i^{(p^{R_i} -1)}v_{i+1}\Big)^p\\
&={(\partial_{x_i}^{p^{m}})}^p+\big(\ad \partial_{x_i}^{p^m} \big)^{p-1}
          \Big( x_i^{(p^{S_i}-p^m)} y_i^{(p^{R_i} -1)}v_{i+1}\Big)\\
&=\partial_{x_i}^{p^{m+1}}+
          x_i^{(p^{S_i}-{p^{m+1}})} y_i^{(p^{R_i} -1)}v_{i+1} ,\qquad\qquad 0\le m< S_i.\qedhere
\end{align*}
\end{proof}

\begin{Lemma}\label{L_clover_relations}
Let $\ch K=p>0$ and a tuple $\Xi$ be fixed. Consider the clover restricted Lie algebra $\TT(\Xi)=\Lie_p(v_0,w_0,u_0)$. Then
\begin{enumerate}
\item
$v_i^{p^{S_{i}}} = y_{i}^{(p^{R_{i}}-1)} v_{i+1}$,\qquad\
$w_i^{p^{R_{i}}} = x_{i}^{(p^{S_{i}}-1)} w_{i+1}$,\qquad\
$u_i^{p^{R_{i}}} = x_{i}^{(p^{S_{i}}-1)} u_{i+1}$,\ for all $i\ge 0$.
\item
 $ [ w_i^{p^{R_i}-1},v_i^{p^{S_i}}]=v_{i+1}$,\qquad
 $ [ v_i^{p^{S_i}-1},w_i^{p^{R_i}}]=w_{i+1}$,\qquad
 $ [ v_i^{p^{S_i}-1},u_i^{p^{R_i}}]=u_{i+1}$,\ for all $i\ge 0$.
\item $v_i,w_i,u_i\in \TT(\Xi)$, $i\ge 0$.
\end{enumerate}
\end{Lemma}
\begin{proof}
Follows from computations of~\cite{Pe20flies}. But let us check the formulas directly.
The first claim is a partial case of Lemma~\ref{Lrelations}.
The third claim follows from the second.
Finally, let us check the second claim.
\begin{align*}
&[ w_i^{p^{R_i}-1},v_i^{p^{S_i}}]
=(\ad w_i)^{p^{R_i}-2}[w_i,v_i^{p^{S_i}}]\\
&\quad =(\ad w_i)^{p^{R_i}-2}
\Big[\dd_{y_i}+y_{i}^{(p^{R_{i}}-1)} x_{i}^{(p^{S_{i}}-1)} w_{i+1}, y_{i}^{(p^{R_{i}}-1)} v_{i+1} \Big]=v_{i+1}.
\qedhere
\end{align*}
\end{proof}

\begin{Lemma}
The subalgebra of $\TT(\Xi)$ generated by $v_0,w_0$ is isomorphic to $\LL(\Xi)$ defined by~\eqref{aibip}.
\end{Lemma}
\begin{proof}
We observe that $v_0,w_0$ have the same presentation as $a_0,b_0\in \LL(\Xi)$.
\end{proof}

\subsection{Head elements of two types}
We construct a clear monomial basis for $\TT(\Xi)$ similar
to that for its subalgebra $\LL(\Xi)\cong \Lie_p(v_0,w_0)\subset \TT(\Xi)$ found in~\cite{Pe17}.
Consider products of two pivot elements of the same length~\eqref{pivot-3}:
\begin{equation} \label{abp}
\begin{split}
h_{i+1}:=[w_i,v_i] &=x_i^{(p^{S_i}-1)} y_i^{(p^{R_i}-2)} v_{i+1}
          -x_i^{(p^{S_i}-2)} y_i^{(p^{R_i}-1)} w_{i+1},\\
g_{i+1}:=[v_i,u_i]&=x_i^{(p^{S_i}-2)} z_i^{(p^{R_i}-1)} u_{i+1}, 
\\
[w_i,u_i]&=0,\qquad i\ge 0.
\end{split}
\end{equation}

\begin{Lemma}[{\cite{Pe17}, Lemma~4.2}]\label{Lcomm}
For all $i\ge 0$ we have the following elements:
\begin{enumerate}
\item
for all $0\le\xi<p^{S_i}$, $0\le\eta<p^{R_i}$
(except the case $\xi=p^{S_i}-1$ and $\eta=p^{R_i}-1$) we get:
\begin{equation}\label{abp2}
h_{i+1}^{\xi,\eta}:=[v_i^{\xi},w_i^{\eta},h_{i+1}]
  =x_i^{(p^{S_i}-1-\xi)} y_i^{(p^{R_i}-2-\eta)} v_{i+1}
  -x_i^{(p^{S_i}-2-\xi)} y_i^{(p^{R_i}-1-\eta)} w_{i+1};
\end{equation}
\item The order of the multiplication above is not essential.
As partial cases, we get:
\begin{align*}
 &h_{i+1}^{0,0}=h_{i+1}=[w_i,v_i];\\
 &h_{i+1}^{p^{S_i}-1,\eta}=y_i^{(p^{R_i}-2-\eta)} v_{i+1},
  \text{ for } 0\le \eta\le  p^{R_i}-2; \text{ as a particular case: }\\
 &h_{i+1}^{p^{S_i}-1,p^{R_i}-2}=v_{i+1};\\
 &h_{i+1}^{\xi,p^{R_i}-1}=- x_i^{(p^{S_i}-2-\xi)} w_{i+1},
  \text{ for } 0\le \xi\le  p^{S_i}-2; \text{ as a particular case: }\\
 &h_{i+1}^{p^{S_i}-2,p^{R_i}-1}=-w_{i+1};
\end{align*}
\end{enumerate}
\end{Lemma}
Thus, for all $i\ge 0$, we obtain elements~\eqref{abp2}, called {\em heads of first type of length} $i+1$:
\begin{equation}\label{heads}
\Big\{h_{i+1}^{\xi,\eta}\ \Big|\ 0\le\xi<p^{S_{i}},\ 0\le\eta<p^{R_{i}},
\text{ except } (\xi=p^{S_{i}}-1 \text{ and } \eta=p^{R_{i}}-1)\Big\}.
\end{equation}
Consider~\eqref{heads} as a table of size $p^{S_i}\times p^{R_i}$, rows and columns being indexed by $\xi$, $\eta$,
the lower right corner is empty.
The table contains $v_{i+1}$ and $-w_{i+1}$ in respective cells.

We multiply~\eqref{abp} by $v_i$, $u_i$ (the order is not essential) we get {\em heads of second type of length} $i+1$:
\begin{equation}\label{heads2}
\Big\{g_{i+1}^{\xi,\zeta}:=[v_i^{\xi},u_i^{\zeta},[v_i,u_i]]
=x_i^{(p^{S_i}-2-\xi)} z_i^{(p^{R_i}-1-\zeta)} u_{i+1} \Big|
\ 0\le\xi\le p^{S_{i}}-2,\ 0\le\zeta\le p^{R_{i}}-1 \Big\}.
\end{equation}
We put~\eqref{heads2} in table of size $(p^{S_i}-1)\times p^{R_i}$,
rows and columns being indexed by $\xi$ and $\zeta$, the low right corner containing $u_{i+1}$.


\subsection{Monomial basis of clover restricted Lie algebras $\TT(\Xi)$}
Define {\it tails} of first and second types:
\begin{equation}\label{rmmp}
\begin{split}
r_n(x,y)&=x_0^{(\xi_0)}y_0^{(\eta_0)}\!\!\cdots x_n^{(\xi_n)}y_n^{(\eta_n)} \in R,\qquad\qquad\qquad
0\le\xi_i<p^{S_i},\ 0\le \eta_i<p^{R_i};\ n\ge 0;\\
r_n(x,y,z)&=x_0^{(\xi_0)}y_0^{(\eta_0)}z_0^{(\zeta_0)}\!\!\cdots x_n^{(\xi_n)}y_n^{(\eta_n)}z_n^{(\zeta_n)} \in R,\qquad
0\le\xi_i<p^{S_i},\ 0\le \eta_i,\zeta_i<p^{R_i};\ n\ge 0.
\end{split}
\end{equation}
For $n<0$ we assume that $r_n=1$.
The notation $r_n(x,y)$ denotes a particular element of type~\eqref{rmmp}.
If appears another element of such type, it
will be denoted as $r_n'(x,y)$, $\tilde r_n(x,y,z)$,  $r_n^{(k)}(*)$, etc.,
where the lower index denotes the biggest index of variables, whose types are given in parenthesis.

Define {\em standard monomials of first type} of length $n$, $n\ge 1$:
\begin{equation}
\label{rmmp3}
\begin{split}
 &r_{n-2}(x,y)h_{n}^{\xi_{n-1},\eta_{n-1}}\\
 &=r_{n-2}(x,y)\Big(x_{n-1}^{(p^{S_{n-1}}-1-\xi_{n-1})} y_{n-1}^{(p^{R_{n-1}}-2-\eta_{n-1})} v_{n}
  -x_{n-1}^{(p^{S_{n-1}}-2-\xi_{n-1})} y_{n-1}^{(p^{R_{n-1}}-1-\eta_{n-1})} w_{n} \Big),\\
 &\quad \text{where }
 0\le\xi_{n-1}<p^{S_{n-1}},\ 0\le \eta_{n-1}<p^{R_{n-1}},
\quad \text{except } (\xi_{n-1}{=}p^{S_{n-1}}{-}1\text{ and }\eta_{n-1}{=}p^{R_{n-1}}{-}1).
\end{split}
\end{equation}
Recall that the {\em heads} $h_{n}^{\xi_{n-1},\eta_{n-1}}$
are described by~\eqref{abp2}, while
the {\em tails} $r_{n-2}(x,y)$ are~\eqref{rmmp}.
We call $x_{n-1},y_{n-1}$ {\it neck letters}.
By Lemma~\ref{Lcomm}, we get the pivot elements $v_n,w_n$, for $n\ge 1$, as particular cases of such monomials.
So, we consider that $v_0,w_0$ are also the standard monomials of first type of length 0.

Define {\em standard monomials of second type} of length $n$, $n\ge 1$:
\begin{equation}
\label{rmmp3B}
\begin{split}
r_{n-2}(x,y,z)g_{n}^{\xi_{n-1},\zeta_{n-1}}=r_{n-2}(x,y,z)
x_{n-1}^{(p^{S_{n-1}}-2-\xi_{n-1})} z_{n-1}^{(p^{R_{n-1}}-1-\zeta_{n-1})} u_{n},\\
\text{ where }\quad
0 \le\xi_{n-1}\le p^{S_{n-1}}{-}2,\quad 0\le \zeta_{n-1}\le p^{R_{n-1}}{-}1.
\end{split}
\end{equation}
Recall that the {\em heads} $g_{n}^{\xi_{n-1},\zeta_{n-1}}$
are described by~\eqref{heads2}, while
the {\em tails} $r_{n-2}(x,y,z)$ are~\eqref{rmmp}.
We call $x_{n-1},z_{n-1}$ {\it neck letters}.
By definition, consider that $u_0$ is a standard monomial of second type of length $0$.

\begin{Theorem}\label{Tstmonom-p}
A basis of the Lie algebra $L=\Lie(v_0,w_0,u_0)$ (i.e. we use only the Lie bracket) is given by
\begin{enumerate}
\item the standard monomials of first type of length $n\ge 0$;
\item the standard monomials of second type of length $n\ge 0$.
\end{enumerate}
\end{Theorem}
\begin{proof}
The standard monomials of first type  form a basis of $\Lie(v_0,w_0)$~\cite[Theorem 5.1]{Pe17}.
Let us prove that the standard monomials of second type belong to $L$.
We proceed by induction on length $n$. We have $u_0\in L$.
By~\eqref{heads2}, we get
$g_{1}^{\xi,\zeta}=x_0^{(p^{S_0}-2-\xi)} z_0^{(p^{R_0}-1-\zeta)} u_{1}=[v_0^{\xi},u_0^{\zeta},[v_0,u_0]]\in R$,
for all $0\le\xi\le p^{S_{0}}-2$, $0\le\zeta\le p^{R_{0}}-1$.
Thus, we have the base of induction for $n=0,1$. Now let $n\ge 1$.
By induction hypothesis, $r_{n-2}(x,y,z)z_{n-1}^{(p^{R_{n-1}}-1)}u_{n}\in  L$.
We use recurrence formula for $v_{n-1}$ and~\eqref{abp}:
\begin{align*}
&[v_{n-1}, r_{n-2}(x,y,z)z_{n-1}^{(p^{R_{n-1}}-1)}u_{n}]
=\Big[\dd_{x_{n-1}}+x_{n-1}^{(p^{S_{n-1}}-1)}y_{n-1}^{(p^{R_{n-1}}-1)}v_{n}, r_{n-2}(x,y,z)z_{n-1}^{(p^{R_{n-1}}-1)}u_{n}\Big]\\
&\qquad\qquad=r_{n-2}(x,y,z)x_{n-1}^{(p^{S_{n-1}}-1)}y_{n-1}^{(p^{R_{n-1}}-1)}z_{n-1}^{(p^{R_{n-1}}-1)}[v_{n},u_{n}]\\
&\qquad\qquad=r_{n-2}(x,y,z)x_{n-1}^{(p^{S_{n-1}}-1)}y_{n-1}^{(p^{R_{n-1}}-1)}z_{n-1}^{(p^{R_{n-1}}-1)}
\cdot x_{n}^{(p^{S_{n}}-2)} z_{n}^{(p^{R_{n}}-1)} u_{n+1}.
\end{align*}
By assumption, $r_{n-2}(x,y,z)$ can have arbitrary powers of its variables.
Multiplying by $v_n=\dd_{x_n}+x_{n}^{(p^{S_{n}}-1)} y_{n}^{(p^{R_{n}}-1)} v_{n+1} $,
we can reduce the power of $x_n$ above to any desired value. The same argument applies to the remaining variables.
Thus, all standard monomials of second type belong to $L$.

It remains to show that products of standard monomials are expressed via standard monomials.
This is true for monomials of first type because they form a basis of $\Lie(v_0,w_0)$~\cite{Pe17}.
Take two standard monomials of second type, shortly written as
$d_1=r_{n-1}u_n$, and $d_2=r_{m-1}u_m$,
divided variables will be shortly written as $x_i^{*}$.
Assume that $n\le m$. Using recurrence presentation, we get
\begin{align*}
[d_1,d_2]&=\big[r_{n-1}(\dd_{z_n}+x_{n}^{*} z_{n}^{*} (\dd_{z_{n+1}}
+\cdots+ x_{m-2}^{*} z_{m-2}^{*}(\dd_{z_{m-1}}+x_{m-1}^{*}z_{m-1}^{*}u_m))),r_{m-1}u_m\big ] \\
&=r_{n-1}\dd_{z_n}(r_{m-1})u_m+r'_{n}\dd_{z_{n+1}}(r_{m-1})u_m+\cdots+r''_{m-2}\dd_{z_{m-1}}(r_{m-1})u_m.
\end{align*}
Observe that we get standard monomials of second type above.

Consider products of monomials of different types.
Write shortly $d_1=r_{n-1}(x,y,z)u_n$,
$d_2=r_{m-2}(x,y)h_{m}^{\xi,\eta}=r'_{m-1}(x,y)v_m-{r}''_{m-1}(x,y)w_m$. 
Consider the case $n\le m$.
Using recurrence presentation and~\eqref{abp},
\begin{align*}
[d_1,d_2]&=\Big[r_{n-1}(x,y,z)\big(\dd_{z_n}+x_{n}^{*} z_{n}^{*} (\dd_{z_{n+1}} +\cdots \\[-3pt]
&\qquad + x_{m-2}^{*} z_{m-2}^{*}(\dd_{z_{m-1}}+x_{m-1}^{*}z_{m-1}^{*}u_m))\big),
r'_{m-1}(x,y)v_m-r''_{m-1}(x,y)w_m\Big ] \\
&=\bar r'_{m-1}(x,y,z)[u_m,v_m]-\bar r''_{m-1}(x,y,z)[u_m,w_m]
=-\bar r'_{m-1}(x,y,z)x_m^{(p^{S_m}-2)} z_m^{(p^{R_m}-1)} u_{m+1},
\end{align*}
yielding  a standard monomial of second type. Consider the case $m<n$. Then
\begin{align*}
d_2&=\sum_{j=m}^{n-1}\Big(r^{(j)}_{j-1}(x,y)\dd_{x_j}- \bar r^{(j)}_{j-1}(x,y)\dd_{y_j}\Big)
+r'_{n-1}(x,y)v_n- r''_{n-1}(x,y)w_n;\\
[d_2,d_1]&=\sum_{j=m}^{n-1}\Big(r^{(j)}_{j-1}(x,y)\dd_{x_j}- \bar r^{(j)}_{j-1}(x,y)\dd_{y_j}\Big)
\Big(r_{n-1}(x,y,z)\Big)u_n +\tilde r'_{n-1}(x,y,z)[v_n,u_n],
\end{align*}
the last term yielding a standard monomial of second type by~\eqref{abp}.
In the preceding sum, the action on variables with indices $n-1$ appears in case $j=n-1$.
Recall that
$d_1=r_{n-1}(x,y,z)u_n=r_{n-2}(x,y,z)x_{n-1}^\a z_{n-1}^\g u_n$, where $0\le\a<p^{S_{n-1}}-1$, $0\le\g<p^{R_{n-1}}$.
After the action in the sum above, we again get  standard monomials of second type.
\end{proof}
By Lemma~\ref{Lrelations},
\begin{equation}\label{powers}
\begin{split}
v_n^{p^i}&=
\begin{cases}
\dd_{x_n}^{p^i}+ x_n^{(p^{S_n}-p^i)} y_n^{(p^{R_n}-1)}v_{n+1}, &\quad  1\le i< S_n;\\
\hfill y_n^{(p^{R_n}-1)} v_{n+1}, &\hfill i=S_n,
\end{cases}\\
w_n^{p^i}&=
\begin{cases}
\dd_{y_n}^{p^i}+ y_n^{(p^{R_n}-p^i)} x_n^{(p^{S_n}-1)}w_{n+1}, &\quad  1\le i< R_n;\\
\hfill x_n^{(p^{S_n}-1)} w_{n+1}, &\hfill i=R_n.
\end{cases}\\
u_n^{p^i}&=
\begin{cases}
\dd_{z_n}^{p^i}+ z_n^{(p^{R_n}-p^i)} x_n^{(p^{S_n}-1)}u_{n+1}, &\quad  1\le i< R_n;\\
\hfill x_n^{(p^{S_n}-1)} u_{n+1}, &\hfill i=R_n.
\end{cases}
\end{split}
\end{equation}
We refer to nonzero powers of $v_n,w_n$ as {\it power standard monomials of first type} of length $n+1$,
powers of $u_n$ are {\it power standard monomials of second type}.
One checks that they are linearly independent with the standard monomials.

\begin{Lemma}\label{Lnum_power}
Let $n\ge 0$. There are $S_n+R_n$ power standard monomials of first type and
$R_n$ power standard monomials of second type of length $n+1$.
\end{Lemma}
\begin{Theorem}\label{Tsemidirect}
Let $\TT(\Xi)=\Lie_p(v_0,w_0,u_0)$ be the clover restricted Lie algebra. Then
\begin{enumerate}
\item
A basis of $\TT(\Xi)$ is given by the standard and power standard monomials of first and second types.
\item We have a semidirect product
$$\TT(\Xi)=\Lie_p(v_0,w_0)\rightthreetimes J ,\qquad \Lie_p(v_0,w_0)\cong \LL(\Xi),$$
the subalgebra $\Lie_p(v_0,w_0)$ is spanned by the standard and power standard monomials
of first type, the ideal $J$ is spanned by the standard and power standard monomials of second type.
\end{enumerate}
\end{Theorem}
\begin{proof}
To get a basis of the $p$-hull of a Lie algebra we need to add $p^m$-powers, $m\ge 1$, of its basis~\cite{StrFar}.
Observe that the standard monomials contain non-trivial tails except for the pivot elements.
Thus, to get a basis of $\Lie_p(v_0,w_0,u_0)$ we add nontrivial powers of the pivot elements, i.e. power standard monomials.
The first claim is proved.

By our specific construction, the pivot elements $\{v_i,w_i|i\ge 0\}$
(i.e., except $u_i$, $i\ge 0$) in~\eqref{pivot-3} are just renaming (in terms of another letters)
of the pivot elements $\{a_i,b_i|i\ge 0\}$~\eqref{aibip} of  Example~\ref{E2}, the latter yielding  the algebra $\LL(\Xi)$.
Their commutators and $p$th powers yield exactly the (power) standard monomials of first type,
which are just renaming of a basis of $\LL(\Xi)$, established in~\cite[Theorem~5.4]{Pe17}.
Thus, the (power) standard monomials of first type span the subalgebra $\Lie_p(v_0,w_0)\cong \LL(\Xi)$.

By computations in proof of Theorem~\ref{Tstmonom-p}, commutators of the standard monomial of second type
with the standard monomials of any type yield monomials of second type, proving that $J$ is an ideal indeed.
The semidirect decomposition of the second claim follows.
\end{proof}

\subsection{Periodic tuple and self-similarity}
The notion of self-similarity for Lie algebras was introduced by Bartholdi~\cite{Bartholdi15}.
A Lie algebra $L$ is called {\it self-similar} if it affords a homomorphism
$\psi : L \to  \Der R \rightthreetimes R \otimes L$,
where $R$ is a commutative algebra and $\Der R$ its Lie algebra of derivations.
A further study of the notion of self-similarity for Lie algebras was done in~\cite{FutKochSis}.
The notion of self-similarity in case of Poisson superalgebras and Jordan superalgebras
is considered in~\cite{PeSh18Jslow}.

Both Lie superalgebras of~\cite{Pe16} are self-similar.
The self-similarity of the two-generated subalgebra
$\LL(\Xi)\cong \Lie_p(v_0,w_0)\subset \TT(\Xi)$ in case of a periodic tuple $\Xi$ was observed in~\cite{Pe17}.
The same phenomena happens in our case as well.

\begin{Lemma}\label{Lself}
Let a tuple $\Xi=(S_i,R_i|i\ge 0)$ be constant or, more generally, periodic.
Then the clover restricted Lie algebra $\TT(\Xi)$ is self-similar.
\end{Lemma}
\begin{proof}
Let $N\in\N$ be the period of the tuple $\Xi$, namely $S_{i+N}=S_i$, $R_{i+N}=R_i$ for $i\ge 0$.
By construction of the pivot elements~\eqref{aibi0} and periodicity,
$$H:=\Lie_p(v_N,w_N,u_N)\cong \Lie_p(v_0,w_0,u_0)= \TT(\Xi) .$$
Consider the subalgebra of divided powers:
$$R_{N-1}:=\big\langle x_0^{(\a_0)}y_0^{(\b_0)}z_0^{(\gamma_0)}\!\!\cdots x_{N-1}^{(\a_{N-1})}y_{N-1}^{(\b_{N-1})}z_{N-1}^{(\gamma_{N-1})}
\big\rangle_K\subset R=R(\Xi),  $$
where bounds for divided powers are the same as in $R$ and determined by $\Xi$.
Using~\eqref{pivot-3} and~\eqref{aibi0}, we get
\begin{multline*}
v_0 = \partial_{x_0} + x_0^{(p^{S_0}-1)}y_0^{(p^{R_0}-1)}\partial_{x_{1}}
       +x_0^{(p^{S_0}-1)}y_0^{(p^{R_0}-1)} x_1^{(p^{S_1}-1)}y_1^{(p^{R_1}-1)}\partial_{x_{2}}+\dots\\
\ldots +x_0^{(p^{S_0}-1)}y_0^{(p^{R_0}-1)}\cdots x_{N-2}^{(p^{S_{N-2}}-1)} y_{N-2}^{(p^{R_{N-2}}-1)} \partial_{x_{N-1}}\\
+ x_0^{(p^{S_0}-1)}y_0^{(p^{R_0}-1)}\cdots x_{N-1}^{(p^{S_{N-1}}-1)} y_{N-1}^{(p^{R_{N-1}}-1)} v_N\\
= d_{N-1}+ p_{N-1} v_N, \qquad\text{where}\quad  d_{N-1}\in\Der R_{N-1}, \quad p_{N-1}\in R_{N-1}.
\end{multline*}
Similar formulas are valid for $w_0$, $u_0$.
These formulas for the generators yield an embedding of self-similarity for the whole of the algebra
\begin{equation*}
\TT(\Xi) \hookrightarrow  \Der R_{N-1} \rightthreetimes R_{N-1} \otimes H, \qquad H=\Lie(v_N,w_N,u_N)\cong\TT(\Xi).
\qedhere
\end{equation*}
\end{proof}

\section{Weight function, $\Z^3$-grading, bounds on weights}

\subsection{Weights}
By {\it pure monomials} we call products of divided powers and pure derivations.
In particular, if a monomial contains one pure derivation, we get a {\it pure Lie monomial}. 
Set
$\a_n=\cwt(\dd_{x_n})=-\cwt(x_n)\in\C$,
$\b_n=\cwt(\dd_{y_n})=-\cwt(y_n)\in\C$,
$\g_n=\cwt(\dd_{z_n})=-\cwt(z_n)\in\C$, for all $n\ge 0$.
This values are easily extended to a weight function on pure monomials, additive on their (Lie or associative) products.
Next, consider weight functions such that all terms in recurrence relation~\eqref{pivot-3} have the same weight,
thus, attaching the same value as a weight for the pivot element as well.
We get a recurrence relation:
\begin{equation} \label{matrix3}
\begin{pmatrix} \a_{n+1}\\ \b_{n+1}\\ \g_{n+1} \end{pmatrix}
=
\begin{pmatrix}
p^{S_{n  }}   & p^{R_{n  }}-1 & 0\\
p^{S_{n  }}-1 & p^{R_{n  }}   & 0\\
p^{S_{n  }}-1 & 0             &p^{R_{n}}
\end{pmatrix}
\begin{pmatrix} \a_{n}\\ \b_{n}\\ \g_{n} \end{pmatrix},
\qquad n\ge 0.
\end{equation}
Recurrence  relation~\eqref{matrix3} expresses weights of the pivot elements of length $n+1$
via weights of the pivot elements of length $n$.
Hence, any weight function satisfying~\eqref{matrix3} is determined by its values on the pivot  elements of zero length,
namely, by $\cwt(v_0), \cwt(w_0),\cwt(u_0)$.

Next, let  $\wt_i(*)$ be a weight function
which is equal to zero for all but $i$-th element in the list $\{v_0,w_0,u_0\}$, for $i=1,2,3$.
Compose the {\it multidegree weight function}
$\Gr(v):=(\wt_1(v),\wt_2(v),\wt_3(v))$, where $v$ is a pure monomial.
By definition, $\Gr(v_0)=(1,0,0)$, $\Gr(w_0)=(0,1,0)$, $\Gr(u_0)=(0,0,1)$.
Thus, the space of weight functions satisfying~\eqref{matrix3} is $3$-dimensional with a basis
$\wt_1(*),\wt_2(*),\wt_3(*)$.
Using~\eqref{matrix3}, we see that $\Gr(v_n)\in\NO^3$ for all $n\ge 0$.
Finally, define the total {\it degree weight function} $\wt(v):=\sum_{j=1}^3 \wt_j(v)$.
Its initial values are $\wt(v_0)=\wt(w_0)=\wt(u_0)=1$.

\begin{Lemma} \label{Lweight_pivo}
$\wt(v_n)=\wt(w_n)=\wt(u_n)=\prod\limits_{i=0}^{n-1}\big(p^{S_{i}}+ p^{R_{i}}-1\big)$, $n\ge 0$,
where $\wt(v_0)=\wt(w_0)=\wt(u_0)=1$.
\end{Lemma}
\begin{proof}
The base of induction is $n=0$, it is trivial.
Assume that the formula holds for $n\ge 0$.
By recurrence relation~\eqref{matrix3}, where  $\a_n=\wt v_n=\b_n=\wt w_n=\gamma_n=\wt u_n=\prod\limits_{i=0}^{n-1}(p^{S_{i}}+ p^{R_{i}}-1)$,
the first row yields
\begin{align*}
\wt v_{n+1}=\a_{n+1}=p^{S_n} \a_n+ (p^{R_n}-1)\b_n 
=(p^{S_n}+p^{R_n}-1)\prod\limits_{i=0}^{n-1}\big(p^{S_{i}}+ p^{R_{i}}-1\big)=
\prod\limits_{i=0}^{n}\big(p^{S_{i}}+ p^{R_{i}}-1\big).
\end{align*}
Using~\eqref{matrix3}, the same formula holds for $\wt w_{n+1}$, $\wt u_{n+1}$.
\end{proof}
\begin{Remark}
The proof shows importance of our peculiar construction~\eqref{pivot-3}.
In particular, we essentially used that
sums of elements of three rows of matrix~\eqref{matrix3} are equal, which originates from the fact that
the top indices for divided powers $y_i,z_i$ are the same, determined by $R_i$,
for $i\ge 0$.
Otherwise it will not be possible to find any formulas for weights of the pivot elements.
\end{Remark}

Recall that $\wt(*)$ is the {\it total degree} function determined by $\wt(v_0)=\wt(w_0)=\wt(u_0)=1$.
Let $\wt_{12}(*)$ be determined by the initial values $\wt_{12}(v_0)=\wt_{12}(w_0)=1$ and $\wt_{12}(u_0)=0$.
Observe that $\wt_{12}(v)=\wt_1(v)+\wt_2(v)$ and $\wt(v)=\wt_{12}(v)+\wt_3(v)$ for any monomial $v$.
Thus, $\wt_{12}(v)$ counts multiplicity of $v$ with respect to $v_0,w_0$ and
$\wt_3(v)$ counts multiplicity with respect to $u_0$ only.
\begin{Lemma} \label{Lweight_pivo2}
For all $n\ge 0$ we have
\begin{align*}
\wt_3(u_n)&=p^{{R_0}+\cdots +R_{n-1}},\qquad   \wt_3(v_n)=\wt_3(w_n)=0;\\
\wt_{12}(v_n)&=\wt_{12}(w_n)=\prod_{i=0}^{n-1}\big(p^{S_{i  }}+ p^{R_{i }}-1\big);\\
\wt_{12}(u_n)&=\prod_{i=0}^{n-1}\big(p^{S_{i  }}+ p^{R_{i  }}-1\big)-
p^{{R_0}+\cdots +R_{n-1}}.
\end{align*}
\end{Lemma}
\begin{proof}
Three formulas are checked by induction. Using $\wt(*)=\wt_{12}(*)+\wt_3(*)$ and Lemma~\ref{Lweight_pivo}, we get  the last formula.
\end{proof}

\subsection{$\NO^3$-gradings}
By a {\it generalized monomial} $a\in\End R$ we call any (Lie or associative) product of
pure monomials and pivot elements.
By construction, actual pivot elements and their products are generalized monomials.
Observe that generalized monomials are written as infinite linear combinations of pure monomials.
Our construction implies that these pure monomials have the same weight;
we call this value the weight of a generalized monomial.
Thus, the weight functions are well-defined on generalized monomials as well.
Also, $\Gr(v)\in\NO^3$ for any generalized monomial $v$.

In many examples studied before~\cite{PeSh09,PeSh13fib,Pe16,Pe17,PeOtto,PeSh18FracPJ}
we were able, as a rule, to compute  explicitly basis functions for the space of weight functions and study multigradings in more details.
Using that base weight functions and multigradings we were able to get more information about our algebras.
In a general setting of the present paper it is not possible.
\begin{Theorem} $\strut$
\label{Tgraded}
\begin{enumerate}
\item
the multidegree weight function $\Gr(v)$ is additive on products of generalized monomials $v,w\in\End R$:
$$ \Gr([v, w])=\Gr(v)+\Gr(w),\qquad \Gr(v\cdot w)=\Gr(v)+\Gr(w). $$
\item
$\TT=\Lie_p(v_0,w_0,u_0)$, $\AA=\Alg(v_0,w_0,u_0)$ are $\NO^3$-graded by multidegree in the generators $\{v_0,w_0,u_0\}$:
$$ \TT=\mathop{\oplus}\limits_{(n_1,n_2,n_3)\in\NO^3} \TT_{n_1,n_2,n_3},\qquad
\AA=\mathop{\oplus}\limits_{(n_1,n_2,n_3)\in\NO^3} \AA_{n_1,n_2,n_3}. $$
\item
$\wt(*)$ counts the degree of $v\in\TT,\AA$ in $\{n_1,n_2,n_3\}$ yielding gradings:
$$ \TT=\mathop{\oplus}\limits_{n=1}^\infty \TT_n,\qquad \AA=\mathop{\oplus}\limits_{n=1}^\infty \AA_n. $$
\end{enumerate}
\end{Theorem}
\begin{proof}
Claim i) follows from the additivity of the weight function on products of pure monomials. Consider ii).
Recall that $\Gr(v)\in\NO^3$ for any generalized monomial $v$ and  $\Gr(*)$ is additive on their products.
Thus, we get $\NO^3$-gradings on $\TT$, $\AA$.

Let $v$ be a monomial in the generators $\{n_1,n_2,n_3\}$ each number  $n_i$ counting entries of $v_0,w_0,u_0$, respectively.
By additivity, $\Gr(v)=n_1\Gr(v_0)+n_2\Gr(w_0)+n_3\Gr(u_0)= (n_1,n_2,n_3)$.
Hence, $\TT$, $\AA$ are $\NO^3$-graded by multidegree in the generators.
Now, the last claim is evident.
\end{proof}

\subsection{Bounds on weights}
\begin{Lemma}[\cite{Pe17}, Lemma 6.5]\label{Lbounds1}
Let $w$ be a (power) standard monomial of first type of length $n\ge 0$. Then
$$ \wt(v_{n-1})+1\le \wt(w)\le\wt(v_n). $$
\end{Lemma}
We shall also write a standard monomial $w$ of second type (see~\eqref{rmmp3B},~\eqref{rmmp}) as:
\begin{equation}\label{second_type2}
\begin{split}
&w=r_{n-3}(x,y,z) x_{n-2}^{(p^{S_{n-2}}-1-\a)} y_{n-2}^{(p^{R_{n-2}}-1-\b)} z_{n-2}^{(p^{R_{n-2}}-1-\g)}\cdot
x_{n-1}^{(p^{S_{n-1}}-2-\xi)}  z_{n-1}^{(p^{R_{n-1}}-1-\zeta)} u_{n},\\
&\text{ where }  0\le\a <p^{S_{n-2}},\ 0\le \b,\gamma<p^{R_{n-2}},\ 0\le \zeta<p^{R_{n-1}};\text { and }\  0\le\xi \le p^{S_{n-1}}-2.
\end{split}
\end{equation}
\begin{Lemma}\label{Lbounds2}
Let $w$ be a (power) standard monomial of second type of length $n\ge 0$.
\begin{enumerate}
\item Let $w$ be a power standard monomial. Then
$$ \wt(v_{n-1})+1\le \wt(w)\le\wt(v_n). $$
\item Let $w$ be a standard monomial of second type of length $n\ge 2$, then
$$ (p^{S_{n-2}}{-}1)\wt(v_{n-2})< \wt(w)\le\wt(v_n). $$
\item
Let $w$ of second type be presented as~\eqref{second_type2}, we get more precise bounds:
\begin{align}\label{estimate}
\wt(w)&> (p^{S_{n-2}}{-}1+\a+\b+\g)\wt (v_{n-2})+(\xi+\zeta)\wt(v_{n-1});\\
\wt(w)&\le\wt(v_{n-1})(\xi+\zeta+2). \label{estimate2}
\end{align}
\item Let $w$ be of second type~\eqref{second_type2} and assume that $\xi>0$ or $\zeta>0$, then
$$ \wt(v_{n-1})+1\le \wt(w)\le\wt(v_n). $$
\end{enumerate}
\end{Lemma}
\begin{proof}
i) Weights of power standard monomials of second type (i.e. powers of $u_{n-1}$)
are equal to weights of powers of $w_{n-1}$, and we apply Lemma~\ref{Lbounds1}.

ii) Let $w$ be a standard monomial of second type~\eqref{rmmp3B} 
Clearly, weight is bounded by $\wt(u_n)=\wt(v_n)$.
We get a lower bound by taking the maximal allowed powers of variables.
Below we get homogeneous components of partial recurrence expansions for $u_0$ and $v_0$, and use that $\wt v_0= \wt u_0=1$.
\begin{align*}
&\wt(w) \ge  \wt \Big(\Big(\prod_{i=0}^{n-2}x_i^{(p^{S_i}-1)}y_i^{(p^{R_i}-1)}z_i^{(p^{R_i}-1)}\Big)
  x_{n-1}^{(p^{S_{n-1}}-2)}z_{n-1}^{(p^{R_{n-1}}-1)} u_n\Big)\\
&\qquad =\wt \Big(\big( \prod_{i=0}^{n-1} x_i^{(p^{S_i}-1)}z_i^{(p^{R_i}-1)}\big)u_n\Big)-\wt(x_{n-1}^{(1)})
    +\wt\Big(\prod_{i=0}^{n-2 }y_i^{(p^{R_i}-1)}\Big) \\
&\qquad =\wt u_0+\wt(v_{n-1})
      +\wt\Big(\prod_{i=0}^{n-2 }y_i^{(p^{R_i}-1)}\Big)
      =1+\wt\Big(\big( \prod_{i=0}^{n-2} x_i^{(p^{S_i}-1)}y_i^{(p^{R_i}-1)}\big)v_{n-1}\Big)
      -\sum_{i=0}^{n-2}\wt(x_i^{(p^{S_i}-1)}) \\[-3pt]
&\qquad =1+\wt(v_0)+\sum_{i=0}^{n-2} (p^{S_i}-1)\wt(v_i)
      \ge 2+(p^{S_{n-2}}-1)\wt (v_{n-2}).
\end{align*}

iii).
The preceding lower bound is given by the maximal allowed powers of the divided variables.
In comparison with that bound we get additional terms $(\a+\b+\g) \wt(v_{n-2})$ and $(\xi+\zeta)\wt(v_{n-1})$.
Recall that by~\eqref{heads2} the head of $w$ is $g_n^{\xi,\zeta}=[v_{n-1}^{\xi},u_{n-1}^{\zeta},[v_{n-1},u_{n-1}]]$,
the latter multiplicands having the same weight, we get $\wt(g_n^{\xi,\zeta})=\wt(v_{n-1})(\xi+\zeta+2)$.
Since tail variables only decrease the weight, we obtain~\eqref{estimate2}.

iv). We use iii) and $(\xi+\zeta)\wt(v_{n-1})\ge \wt(v_{n-1})$.
\end{proof}

\section{Growth of general clover restricted Lie algebras $\TT(\Xi)$}

\subsection{Arbitrary tuple $\Xi$}
\begin{Lemma} \label{Lest}
Fix numbers $p>1$ and $r,s>0$.
Then $p^s+p^r-1>p^{(s+2r)/3}.$
\end{Lemma}
\begin{proof}
Assume that $s\ge r$, then $s\ge (s+2r)/3\ge r$ and
$p^s+p^r-1\ge p^{(s+2r)/3}+(p^r-1)>p^{(s+2r)/3}$.
Consider the case $s<r$, then $s< (s+2r)/3< r$ and
$p^s+p^r-1> p^{(s+2r)/3}+(p^s-1)>p^{(s+2r)/3}$.
\end{proof}

\begin{Theorem}\label{Tgrowth3}
Let $\Xi$ be an arbitrary tuple of parameters and $\TT(\Xi)$ the respective clover restricted Lie algebra.
Then $1\le \GKdim \TT(\Xi)\le 3$.
\end{Theorem}
\begin{proof}
By Theorem~\ref{Tsemidirect}, $\TT(\Xi)$ is a semidirect product of $\LL(\Xi)$ with the ideal $J$,
where bases of $\LL(\Xi)$ and $J$ consist of monomials of first and second types, respectively.
By~\cite[Theorem 7.2]{Pe17}, $\GKdim \LL(\Xi)\le 2$,
yielding an upper bound on the number of the (power) standard monomials of first type.
Since power standard monomials of second type (i.e. powers of $u_n$, $n\ge 0$) behave like powers of $w_n$,
the same estimate on growth of $\LL(\Xi)$ applies to them.

Fix a number $m>1$. It remains to derive an upper bound on the number of standard monomials of second type of weight at most $m$.
Let $n=n(m)$ be such that
\begin{equation}\label{wtan-1}
\wt(v_{n-1})< m\le \wt(v_n).
\end{equation}
Put $m_0:=\wt(v_{n-1})$ and $m_1:=[m/m_0]$.
By~\eqref{Lweight_pivo} and Lemma~\ref{Lest},
\begin{equation}\label{m0}
m_0=\wt(v_{n-1})=\prod_{i=0}^{n-2}(p^{S_i}+p^{R_i}-1)> p^{(S_0+\cdots+ S_{n-2}+2(R_0+\cdots+R_{n-2}))/3},\quad n\ge 2.
\end{equation}
Let $w$ be a standard monomial of second type of length $n'$ and $\wt(w)\le m$. Assume that $n'\ge n+2$.
By Claim ii) of Lemma~\ref{Lbounds2},
$m\ge \wt (w)>\wt(v_{n'-2})\ge \wt(v_n)$, a contradiction with~\eqref{wtan-1}.
Hence, $w$ is of length at most $n+1$.

1) We evaluate a number $f_1(m)$ of standard monomials $w$ of second type of length $n+1$ satisfying $\wt(w)\le m$.
By claim iv) of Lemma~\ref{Lbounds2}, $\xi=\zeta=0$
(i.e. the neck variables reach the maximal values). Thus, we get monomials
\begin{equation}\label{tails2}
w=r_{n-2}(x,y,z) x_{n-1}^{(p^{S_{n-1}}-1-\a)} y_{n-1}^{(p^{R_{n-1}}-1-\b)} z_{n-1}^{(p^{R_{n-1}}-1-\g)}\cdot
x_{n}^{(p^{S_{n}}-2)}  z_{n}^{(p^{R_{n}}-1)} u_{n+1}.
\end{equation}
Using~\eqref{rmmp} and~\eqref{m0}, we estimate a number of tails $r_{n-2}(x,y,z)$ in~\eqref{tails2} as:
\begin{equation}\label{tails}
p^{S_0+\cdots+ S_{n-2}+2(R_0+\cdots+R_{n-2})}< m_0^3.
\end{equation}
Using estimate~\eqref{estimate} (the indices are shifted by one!),
$ m\ge \wt(w)> (1+\a+\b+\g)\wt (v_{n-1})$.  We get estimates
\begin{equation}\label{xi_eta}
0\le \a+\b+\g \le \Big[\frac{m}{\wt(v_{n-1})}\Big]-1= m_1-1,\qquad \a,\b,\g\ge 0.
\end{equation}
A number of possibilities for variables with indices
$n-1$ in~\eqref{tails2} is bounded by a number of triples of integers $\a,\b,\g$
satisfying~\eqref{xi_eta}, which is equal to
\begin{equation}\label{headss}
\binom {m_1+2}{3}=\frac{m_1(m_1+1)(m_1+2)}6\le m_1^3,
\end{equation}
where the last estimate is checked directly for all $m_1\ge 1$.
Using~\eqref{tails} and~\eqref{headss}, we get
\begin{equation*}
f_1(m)< m_0^3 m_1^3 = (m_0m_1)^3\le  m^3.
\end{equation*}

2) Let $f_2(m)$ be a number of standard monomials of second type~\eqref{second_type2}
of length $n$ satisfying $\wt (w)\ \le m$.
Using~\eqref{m0}, a number of possibilities for divided powers with indices
$0,\ldots,n-2$ in~\eqref{second_type2} is evaluated by
\begin{equation}\label{tails7}
p^{S_0+\cdots+ S_{n-2}+2(R_0+\cdots+R_{n-2})}< m_0^3.
\end{equation}
Using estimate~\eqref{estimate},
$ m\ge \wt(w)> (\xi+\zeta)\wt (v_{n-1}).$
We get estimates
\begin{equation}\label{xizeta8}
0\le \xi+\zeta \le \Big[\frac{m}{\wt(v_{n-1})}\Big]= m_1,\qquad \xi,\zeta\ge 0.
\end{equation}
A number of possibilities for the neck letters $x_{n-1},z_{n-1}$
in~\eqref{second_type2} is bounded by a number of pairs of integers $\xi,\zeta$
satisfying~\eqref{xizeta8}. We get a bound
\begin{equation}\label{heads9}
\binom {m_1+2}{2}\le 3 m_1^3, \qquad m_1\ge 1,
\end{equation}
where one checks the last estimate directly.
Using~\eqref{tails7} and~\eqref{heads9} we obtain an estimate
\begin{equation*}
f_2(m)\le 3m_0^3 m_1^3 \le  3m^3.
\end{equation*}

3) Let $f_3(m)$ be a number of all standard monomials of second type~\eqref{second_type2} of length $n-1$.
Using~\eqref{m0}, a number of possibilities for all divided powers (having indices
$0,\ldots,n-2$) is evaluated by
\begin{equation*}
f_3(m)\le p^{S_0+\cdots+ S_{n-2}+2(R_0+\cdots+R_{n-2})}< m_0^3\le m^3.
\end{equation*}
A similar estimate on the number of standard monomials of second type of length $n-2$ is smaller at least by factor $p^{-3}$
than the estimate above. The same applies to lengths $n-3,\ldots, 0$.
Let $\tilde f_3(m)$ be the number of standard monomials of second type~\eqref{second_type2}
of length  at most $n-1$. We get a bound
$$
\tilde f_3(m)\le \sum_{i=0}^{n-1}p^{-3i}\cdot f_3(m)\le \frac {m^3}{1-p^{-3}}\le \frac 87 m^3<2m^3.
$$
Finally, the obtained bounds yield a desired estimate on the number of standard monomial
of second type and weight at most $m$:
\begin{equation*}
f_1(m)+f_2(m)+\tilde f_3(m)\le 6m^3. \qedhere
\end{equation*}
\end{proof}

\subsection{Periodic and constant tuples $\Xi$}
\begin{Theorem}\label{Tperiod}
Let a tuple $\Xi=(S_i,R_i| i\ge 0)$ be periodic:
$S_{i+N}=S_i$, $R_{i+N}=R_i$ for $i\ge 0$, where $N\in\N$ is fixed.
Denote
$$ \mu:=\prod_{i=0}^{N-1} (p^{S_i}+p^{R_i}-1),\qquad
\sigma:= \sum_{i=0}^{N-1 }(S_i+2R_{i}),\qquad
\lambda:=\frac{\sigma\ln p}{\ln \mu}.$$
Consider the clover restricted Lie algebra $\TT=\TT(\Xi)$.  Then
\begin{enumerate}
\item $\GKdim \TT=\LGKdim \TT=\lambda$.
\item
$C_1 m^{\lambda}< \gamma_\TT (m)< C_2 m^{\lambda}$ for $m\ge 1$, and $C_1$, $C_2$ being positive constants.
\item $\lambda\in [1,3]$.
\end{enumerate}
\end{Theorem}
\begin{proof}
By~\eqref{Lweight_pivo} and periodicity,  $\wt (v_{jN})=\mu^j$, $j\ge 0$.
Fix a number $m>1$.
We choose $n=n(m)$ satisfying
\begin{equation}\label{wtan-A}
\wt(v_{(n-1)N})=\mu ^{n-1}< m\le \wt(v_{nN})=\mu^{n}.
\end{equation}
Then
\begin{equation}\label{nup}
n< \log_{\mu}(m)+1.
\end{equation}

Consider a standard monomial $w$ of second type such that $\wt(w)\le m$.
Assume that $w$ has length $n'\ge nN+2$.
By claim ii) of Lemma~\ref{Lbounds2} $\wt(w)>\wt(v_{n'-2})\ge \wt v_{nN}\ge m$, a contradiction.
Hence, $w$ is of length at most $nN+1$.
Let $f_1(n)$ be a number of standard monomials of second type of length at most $(n+1)N$.
We evaluate their number  using a form of these monomials~\eqref{rmmp3B}, \eqref{rmmp},
periodicity, and~\eqref{nup}:
$$ f_1(n)\le  \prod_{i=0}^{(n+1)N-1} p^{S_i+2R_i}
=p^{\sigma (n+1)}
\le p^{\sigma (\log_{\mu}(m)+2)}
\le p^{2\sigma} m^{\sigma \ln p/ \ln \mu}=p^{2\sigma} m^\lambda. $$

Let $w$ be a standard monomial of first type with $\wt(w)\le m$.
As above, by the lower bound in Lemma~\ref{Lbounds1} and the upper bound in~\eqref{wtan-A}, $w$ is of length at most $nN$.
Let $f_2(n)$ be the number of standard monomials of first type of length at most $nN$.
Similarly,  by~\eqref{rmmp3}, \eqref{rmmp}, and~\eqref{nup}, we get
$ f_2(n)< p^{\sigma n}$ yielding a smaller bound than above.

Let $f_3(n)$ be the number of all power standard monomials of weight at most $m$.
By the lower estimates of Lemmas~\ref{Lbounds1}, \ref{Lbounds2}, and the upper bound in~\eqref{wtan-A}, they are of length at most $nN$.
We apply Lemma~\ref{Lnum_power} and~\eqref{nup};
$$
f_3(n)\le \sum_{i=0}^{nN-1}(S_i+2R_i)=  n\sigma \le (\log_{\mu}(m)+1)\sigma.
$$
Now, the upper bound follows using that $\gamma_\TT(m)\le f_1(n)+f_2(n)+f_3(n)$.

By the upper bound in~\eqref{wtan-A}, $n\ge \log_{\mu}(m)$.
Consider standard monomials $w$ of second type~\eqref{rmmp3B} of length $(n-1)N$.
By the lower bound in~\eqref{wtan-A},  $\wt(w)\le \wt(v_{(n-1)N}) <m$.
We evaluate
the number of different parts of their tails $r_{(n-1)N-2}(x,y,z)$ (see~\eqref{rmmp3B}, \eqref{rmmp}), yielding a lower bound:
\begin{equation*}
\gamma_\TT(m)\ge \prod_{i=0}^{(n-3)N-1} p^{S_i+2R_i}
= p^{\sigma (n-3)}\ge p^{\sigma (\log_\mu (m)-3)}
=p^{-3\sigma} m^{{\sigma \ln p}/ \ln \mu }=p^{-3\sigma} m^\lambda. 
\end{equation*}
The second claim follows by our estimates.
The last claim follows from Theorem~\ref{Tgrowth3}.
\end{proof}

\begin{Corollary}\label{Cconstant}
Let a tuple $\Xi=(S_i,R_i| i\ge 0)$ be constant: $S_i=S$, $R_i=R$ for $i\ge 0$, where $S,R\in\N$.
Denote
$\displaystyle \lambda=\frac{(S+2R)\ln p}{\ln(p^S+p^R-1)}. $
Consider the clover restricted Lie algebra $\TT=\TT(S,R):=\TT(\Xi)$.  Then
$\GKdim \TT=\LGKdim \TT=\lambda$.
\end{Corollary}

\begin{Corollary}\label{Cinterval}
Let $\ch K=p>0$.
Consider the self-similar clover restricted  Lie algebras $\TT(S,R)$ given by constant tuples $\Xi$
determined by two integers $S,R$.
Then $\{\GKdim\TT(S,R)\mid S,R\in\N\}$ is dense on $[1,3]$.
\end{Corollary}
\begin{proof}
By choosing the numbers $R=S$ sufficiently large, we can obtain
$\GKdim\TT(S,R)$ arbitrarily close to 3.
By fixing $R=1$ and choosing $S$ sufficiently large we can obtain
$\GKdim\TT(S,R)$ arbitrarily close to 1.

Consider a large positive integer $S$ and $R\in\{1,\ldots,S\}$.
One checks that for $R$, $R'=R+1$ the respective Gelfand-Kirillov dimensions
differ by $O(1/S)$, $S\to+\infty$.
\end{proof}

\section{Proof of main results: Clover restricted Lie algebras of Quasi-linear growth}

\begin{proof}[Proof of Theorem~\ref{Tparam}]
All claims follow from the first one.
Recall that we consider the tuple of integers $\Xi_\kappa:=(S_i:=[(i+1)^{1/\kappa-1}], R_i:=1\mid i\ge 0)$,
and the clover Lie algebra $\TT=\TT(\Xi_\kappa)$.
By Theorem~\ref{Tsemidirect}, $\TT(\Xi)$ is a semidirect product of $\LL(\Xi)$ with the ideal $J$,
where bases of $\LL(\Xi)$ and $J$ consist of monomials of first and second types, respectively.

We start with general estimates used in the proof of the next theorem as well.
By Lemma~\ref{Lweight_pivo},
\begin{align}\label{boundSlower}
\wt(v_n)&=\prod_{i=0}^{n-1}(p^{S_i}{+}p{-}1)>p^{S_0+\cdots+S_{n-1}}, \quad n\ge 1;\\
\wt(v_n)&=\prod_{i=0}^{n-1}(p^{S_i}{+}p{-}1)<\theta p^{S_0+\cdots+S_{n-1}}, \quad
\theta:=\prod_{i=0}^\infty (1{+}p^{1-S_i}), \quad n\ge 1.
\label{boundSupper}
\end{align}
Indeed, it is well known that convergence of the infinite product is equivalent to convergence of the sum
$\sum_{i=0}^\infty p^{1-S_i}$.
We have $S_i>2\log_p i$ for $i\ge N$.
Thus, $\sum_{i=N}^\infty p^{-S_i}\le \sum_{i=N}^\infty 1/i^2<\infty$.
We shall use the following well-known estimates:
\begin{equation}\label{boundsC}
(\kappa+o(1))  n^{1/\kappa} =\sum_{i=0}^{n- 2}S_i
<  \sum_{i=0}^{n}(i+1)^{1/\kappa-1} =(\kappa+o(1))  n^{1/\kappa},\qquad n\to\infty.
\end{equation}

Let us prove the desired upper bound on the standard monomials of second type.
Fix a number $m>1$. Choose $n=n(m)$ such that
\begin{equation}\label{an1an}
\wt(v_{n-1})< m\le \wt(v_n).
\end{equation}
Put $m_0:=\wt(v_{n-1})$ and $m_1:=[m/m_0]$.
By~\eqref{boundSlower} and lower estimate in~\eqref{boundsC}, we get
\begin{align}\label{m22}
m_0&=\wt(v_{n-1})>p^{S_0+\cdots+S_{n-2}}
  \ge p^{(\kappa+o(1)) n^{1/\kappa}},\quad  n\to\infty;\\
\label{llpm}
n &\le  \Big( (1/\kappa{+}o(1))\log_p m_0\Big)^{\kappa}
\le  \Big(\frac {1{+}o(1)}{\kappa\ln p}\ln m\Big)^{\kappa}, \qquad m\to \infty.
\end{align}
Let $w$ be a standard monomial of second type with $\wt(w)\le m$. Suppose that it has length $n'\ge n+2$.
By claim ii) of Lemma~\ref{Lbounds2}, $\wt(w)> \wt(v_{n'-2})\ge \wt(v_n)\ge m\ge \wt(w)$.
The contradiction proves that  $w$ is of length at most $n+1$.

1) We evaluate a number $f_1(m)$ of standard monomials $w$ of second type of length $n+1$ satisfying $\wt(w)\le m$.
By Claim~iv) of Lemma~\ref{Lbounds2}, the head variables in $w$ have the maximal degrees and we get monomials of the form:
\begin{equation}\label{tails2a}
w=r_{n-2}(x,y,z) x_{n-1}^{(p^{S_{n-1}}-1-\a)} y_{n-1}^{(p^{R_{n-1}}-1-\b)} z_{n-1}^{(p^{R_{n-1}}-1-\g)}\cdot
x_{n}^{(p^{S_{n}}-2)}  z_{n}^{(p^{R_{n}}-1)} u_{n+1}.
\end{equation}
Using~\eqref{m22}, we evaluate the number of tails $r_{n-2}(x,y,z)$ in~\eqref{tails2a} as:
\begin{equation} \label{tails3}
p^{S_0+\cdots+S_{n-2}}p^{2(R_0+\cdots +R_{n-2})} < m_0 p^{2n}.
\end{equation}
By estimate~\eqref{estimate}, $ m\ge \wt(w)> (1+\a+\b+\g)\wt (v_{n-1})$.  We get estimates
\begin{equation}\label{xi_eta2}
0\le \a+\b+\g \le \Big[\frac{m}{\wt(v_{n-1})}\Big]-1= m_1-1,\qquad\text{where} \ 0\le \b,\g<p^{R_{n-1}}=p.
\end{equation}
So, both $\b,\g$ have at most $p$ choices.
Now the number of possibilities for variables with indices
$n-1$ in~\eqref{tails2a} is equal to the number of integers $\a,\b,\g$
satisfying~\eqref{xi_eta2}, which is bounded by $p^2 m_1$.
Combining with the bound on the number of tails~\eqref{tails3}, we get
\begin{equation}\label{f1a}
f_1(m)\le p^2m_1\cdot m_0 p^{2n}\le p^2 m p^{2n}.
\end{equation}
2) We evaluate a number $f_2(m)$ of standard monomials of second type and length $n$ such that $\wt(w)\le m$.
Using~\eqref{m22}, a number of possibilities for variables with indices
$0,\ldots,n{-}2$ in~\eqref{second_type2} is evaluated by:
\begin{equation}\label{tails8}
p^{S_0+\cdots+ S_{n-2}+2(R_0+\cdots+R_{n-2})}< m_0 p^{2n}.
\end{equation}
Using estimate~\eqref{estimate}, we get
$ m\ge \wt(w)> (\xi+\zeta)\wt (v_{n-1})$ and
\begin{equation}\label{xizeta}
0\le \xi+\zeta \le \Big[\frac{m}{\wt(v_{n-1})}\Big]= m_1, \qquad\text{where} \ 0\le \zeta <p^{R_{n-1}}=p.
\end{equation}
Thus, the number of possibilities for the neck letters $x_{n-1},z_{n-1}$
in~\eqref{second_type2} is bounded by the number of integers $\xi,\zeta$
satisfying~\eqref{xizeta}, which is bounded by $p(m_1+1)$.
Using the bound on the number of tails~\eqref{tails8}, we get
\begin{equation}\label{f2}
f_2(m)\le p(m_1+1)\cdot m_0 p^{2n} \le 2p\cdot m p^{2n}.
\end{equation}

3) We evaluate a number  $f_3(m)$  of standard monomials of second type and length $n-1$.
Using~\eqref{m22}, a number of possibilities for all divided powers, now having indices
$0,\ldots,n-2$ in~\eqref{second_type2} is evaluated by
\begin{equation*}
f_3(m)\le p^{S_0+\cdots+ S_{n-2}+2(R_0+\cdots+R_{n-2})}< m_0 p^{2n}\le mp^{2n}.
\end{equation*}
The number of standard monomials of second type of length $n-2$ is smaller at least by factor $p^{-3}$
than estimate above. The same applies to lengths $n-3,\ldots, 0$.
Let $\tilde f_3(m)$ be the number of standard monomials of second type~\eqref{second_type2} of length  at most $n-1$.
Using~\ref{llpm}, we get
\begin{equation}\label{bound3}
\tilde f_3(m)\le \sum_{i=0}^{n-2}p^{-3i} f_3(m)\le \frac {m p^{2n}}{1-p^{-3}} \le 2m p^{2n}.
\end{equation}
Let $f_4(m)$ be the number of power standard monomials $w$ of second type of weight at most $m$.
By~\eqref{an1an} and claim i) of Lemma~\ref{Lbounds2}, $w$ is of length at most $n$.
By~\eqref{powers}, $f_4(m)\le R_0+\cdots+R_{n-1}= n$.
Combining~\eqref{f1a}, \eqref{f2}, \eqref{bound3}, and using~\eqref{llpm},
the number of all standard monomials of second type and weight at most $m$ is evaluated by
\begin{align}\label{upper}
&f_1(m)+f_2(m)+\tilde f_3(m)+f_4(m)\le (p^2+2p+2)mp^{2n}+n \\
&\quad \le (p^2+2p+2) m \exp \bigg( 2\ln p
\Big(\frac {1/\kappa{+}o(1)}{\ln p}\ln m\Big)^{\kappa}\bigg)\nonumber\\
&\quad = m\exp \bigg(\frac { 2(\ln p)^{1-\kappa}{+}o(1)}{\kappa^\kappa}(\ln m)^\kappa\bigg),\qquad m\to\infty.
\nonumber
\end{align}

It remains to obtain an upper bound on the number of standard monomials of first type,
these monomials being a basis  of the subalgebra $\LL(\Xi)$, which growth was estimated in~\cite[Theorem 9.2]{Pe17}.
That result has the upper and lower bounds with different constants $C_1$, $C_2$.
Now we are proving a stronger asymptotic with bounds $C+o(1)$, constant being the same for both sides.
A reader can trace and modify that computations or
mimic ideas of more lengthy computations for monomials of second type above using~\eqref{boundsC}, \eqref{llpm} and obtain a bound having actually
a smaller constant (because $z_i$s do not appear in tails resulting in less number of possibilities),
not changing the upper bound given by monomials of second type established above.

Finally, let  us establish the lower bound. We keep notations~\eqref{an1an}.
Similar to~\eqref{boundSupper}
\begin{equation}\label{pCp}
m\le  \wt(v_{n})\le  \prod_{i=0}^{n-1}(p^{(i+1)^{1/\kappa-1}}\!\!\!+p{-}1)
<\theta\prod_{i=0}^{n-1} p^{(i+1)^{1/\kappa-1}}, \quad \theta:=\prod_{i=0}^\infty (1{+}p^{1-(i+1)^{1/\kappa-1}}).
\end{equation}
Using~\eqref{pCp} and the upper bound~\eqref{boundsC}, we get
$m\le  \theta p^{(\kappa+o(1))  n^{1/\kappa}}$. Hence
\begin{equation}\label{ll-k}
n\ge  \bigg(\frac{\log_p (m/\theta)}{\kappa+o(1)}\bigg)^{\kappa}
\ge \Big(\frac {1{+}o(1)}{\kappa\ln p}\ln m\Big)^{\kappa}, \qquad m\to \infty.
\end{equation}
By~\eqref{boundSupper},
\begin{equation} \label{pCp2}
m_0=  \wt(v_{n-1})=  \prod_{i=0}^{n-2}(p^{S_i}+p-1)
<\theta p^{S_0+\cdots+ S_{n-2}}.
\end{equation}

Consider standard monomials of second type
$w=r_{n-2}(x,y,z)g_{n}^{\xi_{n-1},\zeta_{n-1}}$ ~\eqref{rmmp3B} of length $n$.
We evaluate the number of their tails $r_{n-2}(x,y,z)$ using~\eqref{pCp2}
\begin{equation}\label{tailsS}
p^{S_0+\cdots+S_{n-2}}p^{2(R_0+\cdots+R_{n-2})} > \frac {m_0}{\theta}p^{2(n-1)}.
\end{equation}
By our construction and Lemma~\ref{Lweight_pivo}
\begin{equation}\label{boundszeta}
m_1=\Big[\frac m {m_0}\Big ]\le \frac {\wt(v_n)}{\wt(v_{n-1})}=p^{S_{n-1}}+p-1.
\end{equation}
Consider standard monomials of second type $w$ which heads satisfy
$\xi_{n-1}\in\{0,\dots,m_1-p\}$, $\zeta_{n-1}=0$.
Using~\eqref{boundszeta}, we have $0\le \xi_{n-1}<p^{S_{n-1}}$, so, we get standard monomials indeed.
Also, using~\eqref{estimate2}, these monomials are of weight not exceeding $m$:
$$\wt(w)\le \wt(v_{n-1})(\xi_{n-1}+\zeta_{n-1}+2)\le \wt(v_{n-1}) m_1=m_0m_1\le m.$$
There are $m_1-p+1$ such heads.
We multiply this number by the number of different tails~\eqref{tailsS},
using estimate~\eqref{ll-k},  we obtain the desired lower bound:
\begin{align}\label{lower}
&(m_1{-}p{+}1) \frac { m_0 }{\theta} p^{2(n-1)}
 \ge  \frac m{2p^2\theta } p^{2n}\\
&\quad \ge  \frac m{2p^2\theta }\cdot p^{2\big (\textstyle \frac {1+o(1)} {\kappa \ln p}  \ln m\big)^\kappa }
= m\exp \bigg(\frac { 2(\ln p)^{1-\kappa}{+}o(1)}{\kappa^\kappa}(\ln m)^\kappa\bigg),\qquad m\to\infty. \qedhere
\nonumber
\end{align}
\end{proof}

\begin{proof}[Proof of Theorem~\ref{Tparam2}]
We use estimates and notations of the previous proof.
It is sufficient to prove the first claim.
Fix the constant  $\lambda:=(\ln p^2)/\kappa \in\R^+$.
Now we consider the tuple $\Xi_{q,\kappa}=(S_i,R_i\mid i\ge 0)$, where
$R_i:=1$ for all $i\ge 0$, and define integers $S_i$ by induction: $S_0=1$ and
\begin{equation}\label{defS}
S_n:=[\exp^{(q)}(\lambda (n+2) )]+1-S_0-\cdots- S_{n-1},\qquad  n\ge 1.
\end{equation}

Let us prove the desired upper bound on the standard monomials of second type.
Fix a number $m>1$. Choose $n=n(m)$ such that
\begin{equation}\label{an1an2}
\wt(v_{n-1})< m\le \wt(v_n).
\end{equation}
Put $m_0:=\wt(v_{n-1})$ and $m_1:=[m/m_0]$.
By~\eqref{boundSlower} and~\eqref{defS} we get
\begin{align*}
m&> m_0=\wt(v_{n-1})>p^{S_0+\cdots+S_{n-2}}
  \ge p^{\exp^{(q)}(\lambda n )};\\
\label{llpm}
n &< \frac 1\lambda \ln^{(q)}\log_p (m);\\
p^{2n}&< \exp\bigg(\frac {\ln p^2}\lambda \ln^{(q)}\log_p (m)  \bigg)
=\Big (\ln^{(q-1)}\log_p (m)\Big)^{(\ln p^2)/\lambda}\\
&= (\ln^{(q)} m)^{\kappa+o(1)},\qquad m\to\infty.
\end{align*}
Using estimate~\eqref{upper} on the number of all standard monomials of second type of weight at most $m$,
we get the desired upper asymptotic on the number of these monomials.
Similar bounds are valid for monomials of first type.

Let  us check the lower bound. We use notations~\eqref{an1an2}.
By estimate~\eqref{boundSupper} and~\eqref{defS}
\begin{align*}
m&\le  \wt(v_{n}) <\theta p^{S_0+\cdots+S_{n-1}}<\theta p^{\exp^{(q)}(\lambda (n+1) )+1};\\
n&> \frac 1\lambda \ln^{(q)}\Big(\log_p (m/\theta)-1\Big)-1= \frac {1+o(1)}\lambda\ln^{(q+1)}(m), \quad m\to\infty; \\
p^{2n}&>  \exp\bigg(\frac {\ln p^2+o(1)}\lambda \ln^{(q+1)}(m)  \bigg)
= (\ln^{(q)} m)^{\kappa+o(1)},\qquad\ m\to\infty.
\end{align*}
Finally, using the lower bound on the number of standard monomials of second type~\eqref{lower} and
the bound above, we obtain the desired lower bound on the growth of $\TT$.
\end{proof}



\end{document}